\documentclass[11pt]{article}

\usepackage{float}
\usepackage{graphicx}
\usepackage{nameref}
\usepackage[shortlabels,inline]{enumitem}
\usepackage{mathtools}
\usepackage{empheq}

\usepackage[shortlabels]{enumitem}
\setlist[enumerate]{nosep}

\usepackage[doc]{optional}
\usepackage{xcolor}
\usepackage[colorlinks=true,
            linkcolor=refkey,
            urlcolor=lblue,
            citecolor=red]{hyperref}
\usepackage{float}
\usepackage{soul}
\usepackage{graphicx}

\definecolor{labelkey}{rgb}{0,0.08,0.45}
\definecolor{refkey}{rgb}{0,0.6,0.0}
\definecolor{Brown}{rgb}{0.45,0.0,0.05}
\definecolor{lime}{rgb}{0.00,0.8,0.0}
\definecolor{lblue}{rgb}{0.5,0.5,0.99}

 \usepackage{mathpazo}

\colorlet{hlcyan}{cyan!30}

\colorlet{hlred}{red!40}

\usepackage{stmaryrd}
\usepackage{amssymb}
\makeatletter
\def\namedlabel#1#2{\begingroup
   \def\@currentlabel{#2}%
   \label{#1}\endgroup
}
\newcommand{\vast}{\bBigg@{4}}
\newcommand{\Vast}{\bBigg@{5}}
\makeatother

\newcommand{\bx}{\ensuremath{\mathbf{x}}}
\newcommand{\bz}{\ensuremath{\mathbf{z}}}

\newcommand{\weakly}{\ensuremath{\:{\rightharpoonup}\:}}

\newcommand{\kkk}{\ensuremath{{k\in{\mathbb N}}}}
\newcommand{\thalb}{\ensuremath{\tfrac{1}{2}}}
\newcommand{\menge}[2]{\big\{{#1}~\big |~{#2}\big\}}
\newcommand{\tmenge}[2]{\{{#1}~\big |~{#2}\}}

\newcommand{\fenv}[1]%
{\ensuremath{\,\overrightarrow{\operatorname{env}}_{#1}}}
\newcommand{\benv}[1]%
{\ensuremath{\,\overleftarrow{\operatorname{env}}_{#1}}}

\newcommand{\scal}[2]{\left\langle{#1},{#2}  \right\rangle}

\newcommand{\RR}{\ensuremath{\mathbb R}}

\newcommand{\NN}{\ensuremath{\mathbb N}}

\newcommand{\ran}{\ensuremath{{\operatorname{ran}}\,}}
\newcommand{\zer}{\ensuremath{\operatorname{zer}}}

\newcommand{\Fix}{\ensuremath{\operatorname{Fix}}}
\newcommand{\Id}{\ensuremath{\operatorname{Id}}}

\newcommand{\tryu}{\ensuremath{T_{\text{\scriptsize Ryu}}}}
\newcommand{\tmt}{\ensuremath{T_{\text{\scriptsize MT}}}}
\newcommand{\tc}{\ensuremath{T_{\text{\scriptsize C}}}}
\newcommand{\fc}{\ensuremath{F_{\text{\scriptsize C}}}}

\newcommand{\bA}{\ensuremath{{\mathbf{A}}}}

\newcommand{\bB}{\ensuremath{{\mathbf{B}}}}

{\begin{list}{}{%
\settowidth{\labelwidth}{\textrm{#1~}}%
\setlength{\leftmargin}{\labelwidth+\labelsep}}}
{\end{list}}
\usepackage{amsthm}
\usepackage[capitalize,nameinlink]{cleveref}
\crefname{figure}{Figure}{Figures}
\crefname{equation}{}{equations}
\crefname{chapter}{Appendix}{chapters}
\crefname{item}{}{items}
\crefname{enumi}{}{}

\theoremstyle{definition}
\newtheorem{theorem}{Theorem}[section]
\newtheorem{lemma}[theorem]{Lemma}

\newtheorem{corollary}[theorem]{Corollary}

\newtheorem{example}[theorem]{Example}

\newtheorem{fact}[theorem]{Fact}
\newtheorem{remark}[theorem]{Remark}

\DeclarePairedDelimiter{\parens}{\lparen}{\rparen}

\providecommand{\RR}{\mathbb{R}}

\providecommand{\ran}{\operatorname{ran}}

\newcommand{\fix}{\ensuremath{\operatorname{Fix}}}

\providecommand{\Id}{\operatorname{{ Id}}}

\providecommand{\NN}{\mathbb{N}}

\providecommand{\fix}{\operatorname{Fix}}
\providecommand{\ran}{\operatorname{ran}}

\providecommand{\Id}{\operatorname{Id}}

\providecommand{\zer}{\operatorname{zer}}

\providecommand{\RR}{\mathbb{R}}
\providecommand{\NN}{\mathbb{N}}

\definecolor{myblue}{rgb}{.8, .8, 1}

\allowdisplaybreaks 

\begin{document}

\title{\textsf{
The splitting algorithms by~Ryu, by~Malitsky-Tam, and 
by~Campoy
applied to normal cones of linear subspaces
converge strongly to the projection onto the intersection 
}}

\author{
Heinz H.\ Bauschke\thanks{
Mathematics, University
of British Columbia,
Kelowna, B.C.\ V1V~1V7, Canada. E-mail:
\texttt{heinz.bauschke@ubc.ca}.},~
Shambhavi Singh\thanks{
Mathematics, University
of British Columbia,
Kelowna, B.C.\ V1V~1V7, Canada. E-mail:
\texttt{sambha@student.ubc.ca}.},~ 
and
Xianfu Wang\thanks{
Mathematics, University
of British Columbia,
Kelowna, B.C.\ V1V~1V7, Canada. E-mail:
\texttt{shawn.wang@ubc.ca}.}
}

\date{December 6, 2022} 
\maketitle

\vskip 8mm

\begin{abstract} 
Finding a zero of a sum of maximally monotone operators is a fundamental problem in modern optimization and nonsmooth analysis. 
Assuming that the resolvents of the operators are available, this problem
can be tackled with the Douglas-Rachford algorithm.
However, when dealing with three or more operators, one must 
work in a product space with as many factors as there are operators.
In groundbreaking recent work by Ryu and by Malitsky and Tam, it was 
shown that the number of factors can be reduced by one. 
A similar reduction was achieved recently by Campoy through a clever 
reformulation originally proposed by Kruger. 
All three splitting methods guarantee
\emph{weak} convergence to \emph{some} solution
of the underlying sum problem; strong convergence holds in
the presence of uniform monotonicity.

In this paper, we provide a case study when the operators involved
are normal cone operators of subspaces and the solution set is thus
the intersection of the subspaces. Even though these operators 
lack strict convexity, we show that striking  
conclusions are available in this case: \emph{strong} (instead of weak) 
convergence and
the solution obtained is (not arbitrary but) the \emph{projection onto the intersection}.
To illustrate our results, we also perform
numerical experiments. 
\end{abstract}

{\small
\noindent
{\bfseries 2020 Mathematics Subject Classification:}
{Primary 
41A50,
49M27,
65K05, 
47H05; 
Secondary 
15A10,
47H09,
49M37,
90C25. 
}

{\small
\noindent {\bfseries Keywords:}
best approximation,
Campoy splitting, 
Hilbert space, 
intersection of subspaces, 
linear convergence, 
Malitsky-Tam splitting,
maximally monotone operator, 
nonexpansive mapping, 
resolvent,
Ryu splitting.
}

\section{Introduction}

Throughout the paper, we assume that 
\begin{equation*}
\text{$X$ is a real Hilbert space}
\end{equation*}
with inner product $\scal{\cdot}{\cdot}$ and induced norm $\|\cdot\|$. 
Let $A_1,\ldots,A_n$ be maximally monotone operators on $X$.
(See, e.g., \cite{BC2017} for background on maximally monotone operators.)
One central problem in modern optimization and nonsmooth analysis asks to
\begin{equation}
\label{e:genprob}
\text{find $x\in X$ such that $0\in(A_1+\cdots+A_n)x$.}
\end{equation}
In general, solving \cref{e:genprob} may be quite hard.
Luckily, in many interesting cases, we have access to the 
\emph{firmly nonexpansive resolvents} $J_{A_i} := (\Id+A_i)^{-1}$ 
{\color{black} and their associated \emph{reflected resolvents}
$R_{A_i} = 2J_{A_i}-\Id$} 
which opens the door to employ splitting algorithms to solve 
\cref{e:genprob}. 
The most famous instance is the \emph{Douglas-Rachford algorithm} \cite{DougRach} 
whose importance for this problem was brought to light in the seminal paper 
by Lions and Mercier \cite{LionsMercier}. However, 
the Douglas-Rachford algorithm requires that $n=2$; if $n\geq 3$,
one may employ the Douglas-Rachford algorithm to a reformulation
in the product space $X^n$ \cite[Section~2.2]{Comb09}. 
In recent breakthrough work by Ryu \cite{Ryu}, it was shown 
that for $n=3$ one may formulate an algorithm that works 
in $X^2$ rather than $X^3$. 
We will refer to this method as \emph{Ryu's algorithm}. 
Very recently, Malitsky and Tam proposed in \cite{MT} 
an algorithm for a general $n\geq 3$
that is different from Ryu's and that operators in $X^{n-1}$.
(No algorithms exist in product spaces
featuring fewer factors than $n-1$ factors in a certain technical sense. 
See also \cite{Luis} for an extension of the 
Malitsky-Tam algorithm to handle linear operators.)
We will review these algorithms as well as a recent (Douglas-Rachford based) algorithm 
introduced by Campoy \cite{Campoy} in \cref{sec:known} below.
These three algorithms are known to produce \emph{some} solution 
to \cref{e:genprob} via a sequence that converges \emph{weakly}. 
Strong convergence holds in the presence of uniform monotonicity. 

\emph{The aim of this paper is provide a case study for the situation
when the maximally monotone operators $A_i$ are normal cone operators 
of closed linear subspaces $U_i$ of $X$.}
These operators are not even strictly monotone. 
Our main results show that the splitting algorithms by Ryu,
by Malitsky-Tam, and by Campoy actually produce a sequence 
that converges \emph{strongly}\,! We are also able to \emph{identify the limit} to
be the \emph{projection} onto the intersection
$U_1\cap\cdots\cap U_n$\,!
The proofs of these results rely on the explicit identification of the fixed point
set of the underlying operators.
Moreover, a standard translation technique gives the same result for
\emph{affine} subspaces of $X$ provided their intersection is nonempty.

The paper is organized as follows.
In \cref{sec:aux}, we collect various auxiliary results for later use.
The known convergence results on Ryu splitting, on 
Malitsky-Tam splitting, and on Campoy splitting are reviewed in \cref{sec:known}.
Our main results are presented in \cref{sec:main}.
Matrix representations of the various operators involved
are provided in \cref{sec:matrix}. 
These are useful for our numerical experiments in \cref{sec:numexp}.
{\color{black} We investigate the case of three lines in 
the Euclidean plane in \cref{sec:3lines}.}
Finally, we offer some concluding remarks in \cref{sec:end}.

The notation employed in this paper is 
standard and follows largely \cite{BC2017}. 
When $z=x+y$ and $x\perp y$, then we also write
$z=x\oplus y$ to stress this fact.
Analogously for the Minkowski sum $Z=X+Y$, we
write $Z= X\oplus Y$ as well as $P_Z = P_X\oplus P_Y$ for the 
associated projection provided that $X\perp Y$.

\section{Auxiliary results}
\label{sec:aux}

In this section, we collect useful properties of projection operators
and results on iterating linear/affine nonexpansive operators. 
We start with projection operators.

\subsection{Projections}

\begin{fact}
\label{f:orthoP}
Suppose $U$ and $V$ are nonempty closed convex subsets of $X$ such that
$U\perp V$. Then
$U\oplus V$ is a nonempty closed subset of $X$ and 
$P_{U\oplus V} = P_U\oplus P_V$. 
\end{fact}
\begin{proof}
See \cite[Proposition~29.6]{BC2017}. 
\end{proof}

Here is a well known illustration of \cref{f:orthoP} which
we will use repeatedly in the paper (sometimes without explicit mentioning). 

\begin{example}
\label{ex:perp}
Suppose $U$ is a closed linear subspace of $X$.
Then
$P_{U^\perp} = \Id-P_U$. 
\end{example}
\begin{proof}
The orthogonal complement $V := U^\perp$ satisfies $U\perp V$
and also $U+V=X$; thus $P_{U+V}=\Id$ and the result follows.
\end{proof}

\begin{fact} {\bf (Anderson-Duffin)} 
\label{f:AD}
Suppose that $X$ is finite-dimensional and that 
$U,V$ are two linear subspaces of $X$.
Then
$P_{U\cap V} = 2P_U(P_U+P_V)^\dagger P_V$,
where ``$^\dagger$'' denotes the Moore-Penrose inverse of a matrix.
\end{fact}
\begin{proof}
See, e.g., \cite[Corollary~25.38]{BC2017} or
the original \cite{AD}. 
\end{proof}

\begin{corollary}
\label{c:AD3}
Suppose that $X$ is finite-dimensional and that 
$U,V,W$ are three linear subspaces of $X$.
Then
\begin{equation*}
P_{U\cap V\cap W} = 4P_U(P_U+P_V)^\dagger P_V\big(2P_U(P_U+P_V)^\dagger P_V+P_W\big)^\dagger P_W. 
\end{equation*}
\end{corollary}
\begin{proof}
Use \cref{f:AD} to find $P_{U\cap V}$,
and then use \cref{f:AD} again on $(U\cap V,W)$. 
\end{proof}

\begin{corollary}
\label{c:ADsum}
Suppose that $X$ is finite-dimensional and that $U,V$ are two linear
subspaces of $X$. 
Then
\begin{align*}
P_{U+V} &= \Id-2P_{U^\perp}(P_{U^\perp}+P_{V^\perp})^\dagger P_{V^\perp}\\
&=\Id-2(\Id-P_U)\big(2\Id-P_U-P_V \big)^\dagger(\Id-P_V).
\end{align*}
\end{corollary}
\begin{proof}
Indeed, $U+V = (U^\perp\cap V^\perp)^\perp$ and so 
$P_{U+V} = \Id - P_{U^\perp\cap V^\perp}$.
Now apply \cref{f:AD} to $(U^\perp,V^\perp)$
followed by \cref{ex:perp}. 
\end{proof}

\begin{fact}
\label{f:Pran}
Let $Y$ be a real Hilbert space, and 
let $A\colon X\to Y$ be a continuous linear operator with closed range.
Then 
$P_{\ran A} = AA^\dagger$. 
\end{fact}
\begin{proof}
See, e.g., \cite[Proposition~3.30(ii)]{BC2017}. 
\end{proof}

\subsection{Linear (and affine) nonexpansive iterations}

We now turn results on iterating linear or affine nonexpansive operators.

\begin{fact}
\label{f:convasymp}
Let $L\colon X\to X$ be linear and nonexpansive, 
and let $x\in X$. 
Then 
\begin{equation*}
L^kx \to P_{\Fix L}(x)
\quad\Leftrightarrow\quad
L^kx - L^{k+1}x \to 0. 
\end{equation*}
\end{fact}
\begin{proof}
See \cite[Proposition~4]{Baillon},
\cite[Theorem~1.1]{BBR}, 
\cite[Theorem~2.2]{BDHP}, or 
\cite[Proposition~5.28]{BC2017}. 
(The versions in \cite{Baillon} and \cite{BBR} are much more general.)
\end{proof}

\begin{fact}
\label{f:averasymp}
Let $T\colon X\to X$ be averaged nonexpansive with 
$\Fix T\neq\varnothing$. 
Then $(\forall x\in X)$ 
$T^kx-T^{k+1}x\to 0$. 
\end{fact}
\begin{proof}
See Bruck and Reich's paper \cite{BruRei} or \cite[Corollary~5.16(ii)]{BC2017}.
\end{proof}

\begin{corollary}
\label{c:key}
Let $L\colon X\to X$ be linear and averaged nonexpansive.
Then 
$(\forall x\in X)$ $L^kx \to P_{\Fix L}(x)$.
\end{corollary}
\begin{proof}
Because $0\in\Fix L$, we have $\Fix L\neq\varnothing$. 
Now combine
\cref{f:convasymp} with \cref{f:averasymp}.
\end{proof}

\begin{fact}
\label{f:BLM}
Let $L$ be a linear nonexpansive operator and let $b\in X$.
Set $T\colon X\to X\colon x\to Lx+b$ and suppose that 
$\Fix T\neq\varnothing$.
Then $b\in\ran(\Id-L)$, and for every $x\in X$ and
$a\in(\Id-L)^{-1}b$, the following hold:
\begin{enumerate}
\item $b=a-La\in\ran(\Id-L)$.
\item $\Fix T = a + \Fix L$.
\item 
\label{f:BLMiii}
$P_{\Fix T}(x) = P_{\Fix L}(x) + P_{(\Fix L)^\perp}(a)$.
\item $T^kx = L^k(x-a)+a$.
\item $L^kx \to P_{\Fix L}x$ $\Leftrightarrow$ $T^kx \to P_{\Fix T}x$. 
\end{enumerate}
\end{fact}
\begin{proof}
See \cite[Lemma~3.2 and Theorem~3.3]{BLM}. 
\end{proof}

\begin{remark}
\label{r:BLM}
Consider \cref{f:BLM} and its notation.
If $a\in(\Id-L)^{-1}b$ then 
$P_{(\Fix L)^\perp}(a)$ is likewise because 
$b = (\Id-L)a = (\Id-L)(P_{\Fix L}(a)+P_{(\Fix L)^\perp}(a))
=(\Id-L)P_{(\Fix L)^\perp}(a)$; moreover,
using \cite[Lemma~3.2.1]{Groetsch}, we see that 
\begin{equation*}
(\Id-L)^\dagger b 
=(\Id-L)^\dagger(\Id-L)a
=P_{(\ker(\Id-L))^\perp}(a)
=P_{(\Fix L)^\perp}(a),
\end{equation*}
where again ``$^\dagger$'' denotes the Moore-Penrose inverse of a continuous
linear operator (with possibly nonclosed range).
So given $b\in X$, we may concretely set 
\begin{equation*}
a = (\Id-L)^\dagger b \in (\Id-L)^{-1}b;
\end{equation*}
with this choice, \cref{f:BLMiii} turns into the even more pleasing identity
\begin{equation*}
P_{\Fix T}(x) = P_{\Fix L}(x) + a.
\end{equation*}
\end{remark}

\section{Ryu, Malitsky-Tam, and Campoy splitting }

In this section, we present the precise form of
Ryu's, the Malitsky-Tam, and Campoy's algorithms and review 
known convergence results. 

\label{sec:known}

\subsection{Ryu splitting}

We start with Ryu's algorithm. In this subsection, 
\begin{equation*}
\text{$A,B,C$ are maximally monotone operators on 
$X$, 
}
\end{equation*}
with resolvents $J_A,J_B,J_C$, respectively.
The problem of interest is to 
\begin{equation}
\label{e:Ryuprob}
\text{find $x\in X$ such that $0\in (A+B+C)x$,}
\end{equation}
and we assume that \cref{e:Ryuprob} has a solution. 
The algorithm pioneered by Ryu  \cite{Ryu} provides a method
for finding a solution to \cref{e:Ryuprob}.
It proceeds as follows. Set\footnote{We will express vectors in
product spaces both as column and as row vectors depending 
on which version is more readable.}
\begin{equation}
\label{e:genM}
M\colon X\times X\to X\times X\times X\colon
\begin{pmatrix}
x\\
y
\end{pmatrix}
\mapsto
\begin{pmatrix}
J_A(x)\\[+1mm]
J_B(J_A(x)+y)\\[+1mm]
J_C\big(J_A(x)-x+J_B(J_A(x)+y)-y\big) 
\end{pmatrix}. 
\end{equation}
Next, denote by $Q_1\colon X\times X\times X\to X\colon 
(x_1,x_2,x_3)\mapsto x_1$
and similarly for $Q_2$ and $Q_3$. 
We also set $\Delta := \menge{(x,x,x)\in X^3}{x\in X}$. 
We are now ready to introduce the \emph{Ryu operator}
\begin{equation}
\label{e:TRyu}
T := \tryu \colon X^2\to X^2\colon 
z\mapsto 
z + \big((Q_3-Q_1)Mz,(Q_3-Q_2)Mz\big). 
\end{equation}
Given a starting point $(x_0,y_0)\in X\times X$, 
the basic form of 
Ryu splitting generates a governing sequence via
\begin{equation}
\label{e:basicRyu}
(\forall \kkk)\quad (x_{k+1},y_{k+1}) := 
(1-\lambda)(x_k,y_k) + \lambda T(x_k,y_k).
\end{equation}

The following result records the basic convergence properties 
by Ryu \cite{Ryu}, 
and recently improved by Arag\'on-Artacho, Campoy, and Tam \cite{AACT20}. 

\begin{fact} {\bf (Ryu and also Aragon-Artacho-Campoy-Tam)}
\label{f:Ryu}
The operator
$\tryu$ is nonexpansive with 
\begin{equation}
\label{e:fixtryu}
\Fix \tryu = 
\menge{(x,y)\in X\times X}{J_A(x)=J_B(J_A(x)+y) = J_C(R_A(x)-y)}
\end{equation}
and 
\begin{equation*}
\zer(A+B+C) = J_A\big(Q_1\Fix\tryu\big). 
\end{equation*}
Suppose that $0<\lambda<1$ and 
consider the sequence generated by 
\cref{e:basicRyu}. 
Then there exists $(\bar{x},\bar{y})\in X\times X$ such that 
\begin{equation*}
(x_k,y_k)\weakly (\bar{x},\bar{y}) \in \Fix\tryu, 
\end{equation*}
\begin{equation*}
M(x_k,y_k) \weakly M(\bar{x},\bar{y})\in\Delta, 
\end{equation*}
and 
\begin{equation}
\label{e:210815d}
\big((Q_3-Q_1)M(x_k,y_k),(Q_3-Q_2)M(x_k,y_k)\big)\to (0,0). 
\end{equation}
In particular,
\begin{equation*}
J_A(x_k) \weakly J_A\bar{x}\in \zer(A+B+C).
\end{equation*}
\end{fact}
\begin{proof}
See \cite{Ryu} and \cite{AACT20}. 
\end{proof}

\subsection{Malitsky-Tam splitting}

We now turn to the Malitsky-Tam algorithm. 
In this subsection, let $n\in\{3,4,\ldots\}$ and 
let $A_1,A_2,\ldots,A_n$ 
be maximally monotone operators on  $X$. 
The problem of interest is to 
\begin{equation}
\label{e:MTprob}
\text{find $x\in X$ such that $0\in (A_1+A_2+\cdots + A_n)x$,}
\end{equation}
and we assume that \cref{e:MTprob} has a solution. 
The algorithm proposed by Malitsky and Tam \cite{MT} provides a method
for finding a solution to \cref{e:MTprob}.
Now set\footnote{Again, we will express vectors in
product spaces both as column and as row vectors depending 
on which version is more readable.}
\begin{subequations}
\label{e:MTM}
\begin{align}
M\colon X^{n-1}
&\to 
X^n\colon
\begin{pmatrix}
z_1\\
\vdots\\
z_{n-1}
\end{pmatrix}
\mapsto
\begin{pmatrix}
x_1\\
\vdots\\
x_{n-1}\\
x_n
\end{pmatrix},
\quad\text{where}\;\;
\\[+2mm]
&(\forall i\in\{1,\ldots,n\})
\;\;
x_i = \begin{cases}
J_{A_1}(z_1), &\text{if $i=1$;}\\
J_{A_i}(x_{i-1}+z_i-z_{i-1}), &\text{if $2\leq i\leq n-1$;}\\
J_{A_n}(x_1+x_{n-1}-z_{n -1}), &\text{if $i=n$.}
\end{cases}
\end{align}
\end{subequations}
As before, we denote by $Q_1\colon X^n \to X\colon 
(x_1,\ldots,x_{n-1},x_n)\mapsto x_1$
and similarly for $Q_2,\ldots,Q_n$. 
We also set 
\begin{equation}\label{e:delta}
\Delta := \menge{(x,\ldots,x)\in X^n}{x\in X},
\end{equation} 
which is also known as the diagonal in $X^n$. 
We are now ready to introduce the \emph{Malitsky-Tam (MT) operator}
\begin{equation}
\label{e:TMT}
T := \tmt \colon X^{n-1}\to X^{n-1}\colon 
\bz\mapsto 
\bz+ 
\begin{pmatrix}
(Q_2-Q_1)M\bz\\
(Q_3-Q_2)M\bz\\
\vdots\\
(Q_n-Q_{n-1})M\bz
\end{pmatrix}. 
\end{equation}
Given a starting point $\bz_0 \in X^{n-1}$, 
the basic form of 
MT splitting generates a
governing sequence via
\begin{equation}
\label{e:basicMT}
(\forall \kkk)\quad \bz_{k+1} := 
(1-\lambda)\bz_{k} + \lambda T\bz_k. 
\end{equation}

The following result records the basic convergence. 

\begin{fact} {\bf (Malitsky-Tam)} 
\label{f:MT}
The operator
$\tmt$ is nonexpansive with 
\begin{equation}
\label{e:fixtmt}
\Fix \tmt = 
\menge{z\in X^{n-1}}{Mz \in \Delta},
\end{equation}
\begin{equation*}
\zer(A_1+\cdots+A_n) = J_{A_1}\big(Q_1\Fix\tmt\big).
\end{equation*}
Suppose that $0<\lambda<1$ and 
consider the sequence generated by 
\cref{e:basicMT}. 
Then there exists $\bar{\bz}\in X^{n-1}$ such that 
\begin{equation*}
\bz_k\weakly \bar{\bz} \in \Fix\tmt, 
\end{equation*}
\begin{equation*}
M\bz_{k} \weakly M\bar{\bz}\in\Delta,
\end{equation*}
and 
\begin{equation}
\label{e:210817a}
(\forall (i,j)\in\{1,\ldots,n\}^2)\quad 
(Q_i-Q_j)M\bz_k\to 0. 
\end{equation}
In particular,
\begin{equation*}
Q_1 M\bz_k \weakly Q_1 M\bar{\bz}\in \zer(A_1+\cdots+A_n). 
\end{equation*}
\end{fact}
\begin{proof}
See \cite{MT}. 
\end{proof}

\subsection{Campoy splitting}
\label{ss:Csplit}
Finally, we turn to Campoy's algorithm \cite{Campoy}. 
Again, let $n\in\{3,4,\ldots\}$ and 
let $A_1,A_2,\ldots,A_n$ 
be maximally monotone operators on  $X$. 
The problem of interest is to 
\begin{equation}
\label{e:Cprob}
\text{find $x\in X$ such that $0\in (A_1+A_2+\cdots + A_n)x$,}
\end{equation}
and we assume again that \cref{e:Cprob} has a solution. 
Denote the diagonal in $X^{n-1}$ by $\Delta$, and define
the \emph{embedding operator} $E$ by 
\begin{equation}
\label{e:embed}
E\colon X\to X^{n-1}\colon x\mapsto (x,x,\ldots,x).
\end{equation}
Next, define two operators $\bA$ and $\bB$ on $X^{n-1}$ by
\begin{align*}
\bA\colon (x_1,\ldots,x_{n-1}) &\mapsto \tfrac{1}{n-1}\big(A_nx_1,\ldots,A_nx_{n-1}\big) + N_{\Delta}(x_1,\ldots ,x_{n-1}),\\
\bB\colon (x_1,\ldots,x_{n-1}) &\mapsto \big(A_1x_1,\ldots,A_{n-1}x_{n-1}\big).
\end{align*}
We note that this way of splitting is also contained 
in early works of Alex Kruger
(see \cite{Kruger81} and \cite{Kruger85}) to whom we are grateful for making
us aware of this connection. However, the relevant resolvents (see \cref{f:C}) were only
very recently computed by Campoy. 
We now define the \emph{Campoy operator} by 
\begin{equation}
\label{e:CT}
T := \tc \colon X^{n-1}\to X^{n-1}\colon 
\bz\mapsto 
R_{\bB}R_{\bA}\bz = \bz - 2J_{\bA}\bz + 2J_{\bB}R_{\bA}\bz. 
\end{equation}
Given a starting point $\bz_0 \in X^{n-1}$, 
the basic form of 
Campoy's splitting algorithm generates a
governing sequence via
\begin{equation}
\label{e:basicC}
(\forall \kkk)\quad \bz_{k+1} := 
(1-{\lambda}\big)\bz_{k} + {\lambda} T\bz_k. 
\end{equation}

The following result records basic properties and the convergence result. 

\begin{fact} {\bf (Campoy)} 
\label{f:C}
The operators $\bA$ and $\bB$ are maximally monotone, with resolvents 
\begin{subequations}
\label{e:f:C}
\begin{align}
M := J_{\bA}\colon (x_1,\ldots,x_{n-1})&\mapsto E\bigg(J_{\frac{1}{n-1}A_n}\Big(
\tfrac{1}{n-1}\sum_{i=1}^{n-1}x_i\Big)\bigg),\label{e:f:Ca}\\
J_{\bB} \colon (x_1,\ldots,x_{n-1})&\mapsto 
\big(J_{A_1}x_1,\ldots,J_{A_{n-1}}x_{n-1} \big), 
\end{align}
\end{subequations}
respectively. 
We also have 
\begin{equation*}
\zer(\bA+\bB) = E\big(\zer(A_1+\cdots+A_n)\big).
\end{equation*}
The operator 
\begin{equation}
\label{e:CF}
\fc := \thalb\Id + \thalb\tc
=\Id-J_{\bA}+J_{\bB}R_{\bA} 
\end{equation}
is the standard Douglas-Rachford (firmly nonexpansive) 
operator for 
finding a zero of $\bA+\bB$ and its reflected version is the 
Campoy operator 
$\tc = 2\fc-\Id$ is therefore nonexpansive. 
Suppose that $0<\lambda<1$ and 
consider the sequence generated by 
\cref{e:basicC}. 
Then there exists $\bar{\bz}\in X^{n-1}$ such that 
\begin{equation*}
\bz_k\weakly \bar{\bz} \in \Fix\tc, 
\end{equation*}
\begin{equation*}
M\bz_{k} = J_{\bA}\bz_{k} \weakly J_{\bA}\bar{\bz}\in\zer(\bA+\bB).
\end{equation*}
\end{fact}
\begin{proof}
See \cite[Theorem~3.3(i)--(iii) and Theorem~5.1]{Campoy}, and \cite[Theorem~26.11]{BC2017}.
\end{proof}

If one wishes to avoid the product space reformulation, then one may set
\begin{equation*}
p_k:=J_{\frac{1}{n-1}A_n}\Big(\frac{1}{n-1}\sum_{i=1}^{n-1} z_{k,i}\Big)
\end{equation*}
so that $J_{\bA}(\bz_k)=E(p_k)$. 
Now for $i\in \{1,\dots,n-1\}$, we define
\begin{equation*}
x_{k,i}:=J_{A_i}(2p_k-z_{k,i}).
\end{equation*}
It follows that
\begin{equation*}
\begin{pmatrix}
x_{k,1}\\\vdots\\x_{k,n-1}
\end{pmatrix}=J_{\bB}(R_{\hrulefill\bA}(\bz)).
\end{equation*}
Now one can rewrite Campoy's algorithm as follows (and this is
his original formulation):
Given a starting point $(z_{0,1},\ldots,z_{0,n-1})\in X^{n-1}$ and $k\in\NN$, 
update 
\begin{subequations}
\begin{align}
p_k&=J_{\frac{1}{n-1}A_n}\bigg(\frac{1}{n-1}\sum_{i=1}^{n-1} z_{k,i}\bigg),\\
(\forall i\in \{1,\ldots,n-1\})\qquad 
x_{k,i}&=J_{A_i}(2p_k-z_{k,i}),\\
(\forall i\in \{1,\ldots,n-1\})\quad z_{k+1,i}&=z_{k,i}+\lambda(x_{k,i}-p_k),\label{e:CTalt}
\end{align}
\end{subequations}
which gives us the equivalence of \cref{e:basicC} and \cref{e:CTalt}.

For future use in \cref{sec:matrix}, we bring the Campoy operator into a framework 
similar to the other algorithms. As before, we denote by $Q_1\colon X^{n-1} \to X\colon 
(x_1,\ldots,x_{n-1})\mapsto x_1$
and similarly for $Q_2,\ldots,Q_n$. 
We recall from \cref{e:f:Ca} that  
\begin{subequations}
\begin{align}
\label{e:CampoyM}
M &: 
X^{n-1}\to X^{n-1} : 
\bz = 
\begin{pmatrix}
z_{1}\\
\vdots\\
z_{n-1}
\end{pmatrix}
\mapsto
\begin{pmatrix}
J_{\frac{1}{n-1}A_n}\Big(\frac{1}{n-1}\sum_{i=1}^{n-1}z_{i}\Big)\\
\vdots\\
J_{\frac{1}{n-1}A_n}\Big(\frac{1}{n-1}\sum_{i=1}^{n-1}z_{i}\Big)
\end{pmatrix}, 
\end{align}
and we set
\begin{align}
\label{e:CampoyS}
S&:X^{n-1}\to X^{n-1}:
\bz = \begin{pmatrix}
z_{1}\\
\vdots\\
z_{n-1}
\end{pmatrix}
\mapsto
\begin{pmatrix}
x_{1}\\
\vdots\\
x_{n-1}
\end{pmatrix},
\quad\text{where}\;\;
\\[+2mm]
&(\forall i\in\{1,\ldots,n-1\})
\;\;
x_{i} = J_{A_i}(2Q_iM\bz- z_{i}). 
\end{align}
\end{subequations}
We can then rewrite the \emph{Campoy operator} in \cref{e:CT} as
 \begin{equation}
 \label{e:CampoyT}
 T = \tc \colon X^{n-1}\to X^{n-1}\colon 
 \bz\mapsto 
 \bz+ 2S\bz - 2M\bz. 
 \end{equation}

\section{Main results}

We are now ready to tackle our main results.
We shall find useful descriptions of the fixed point sets
of the Ryu, the Malitsky-Tam, and the Campoy operators.
These description will allow us to deduce strong convergence of the 
iterates to the projection
onto the intersection. 

\label{sec:main}

\subsection{Ryu splitting}

In this subsection, we assume that 
\begin{equation*}
\text{
$U,V,W$ are closed linear subspaces of $X$.
}
\end{equation*}
We set
\begin{equation*}
A := N_U,
\;\;
B := N_V,
\;\;
C := N_{W}. 
\end{equation*}
Then
\begin{equation*}
Z := \zer(A+B+C) = U\cap V\cap W.
\end{equation*}
Using linearity of the projection operators, the operator $M$ defined in \cref{e:genM} turns into
\begin{equation}
\label{e:linM}
M\colon X\times X\to X\times X\times X\colon
\begin{pmatrix}
x\\
y
\end{pmatrix}
\mapsto
\begin{pmatrix}
P_Ux\\[+1mm]
P_VP_Ux+P_Vy\\[+1mm]
P_WP_Ux + \textcolor{black}{P_W}P_VP_Ux-P_Wx \textcolor{black}{+}P_WP_Vy-P_Wy
\end{pmatrix},
\end{equation}
while the Ryu operator is still (see \cref{e:TRyu})
\begin{equation}
\label{e:linT}
T := \tryu \colon X^2\to X^2\colon 
z\mapsto 
z + \big((Q_3-Q_1)Mz,(Q_3-Q_2)Mz\big). 
\end{equation}

We now determine the fixed point set of the Ryu operator.

\begin{lemma}
\label{firelemma}
Let $(x,y)\in X\times X$. 
Then 
\begin{equation}
\label{e:210815b}
\Fix T = \big(Z\times\{0\}\big) \oplus 
\Big(\big(U^\perp\times V^\perp) \cap \big(\Delta^\perp+(\{0\}\times W^\perp) \big)\Big),
\end{equation}
where $\Delta := \menge{(x,x)\in X\times X}{x\in X}$. 
Consequently, setting 
\begin{equation*}
E:=\big(U^\perp\times V^\perp) \cap \big(\Delta^\perp+(\{0\}\times W^\perp) \big),
\end{equation*}
we have 
\begin{equation}
\label{e:210815c}
P_{\Fix T}(x,y) = (P_{Z}x,0)\oplus P_E(x,y) \in (P_Zx\oplus U^\perp)\times V^\perp. 
\end{equation}
\end{lemma}
\begin{proof}
Note that $(x,y)=(P_{W^\perp}y+(x-P_{W^\perp}y),P_{W^\perp}y+P_Wy)
=(P_{W^\perp}y,P_{W^\perp}y)+(x-P_{W^\perp}y,P_Wy)
\in \Delta + (X\times W)$.
Hence 
\begin{equation*}
X\times X = \Delta+(X\times W)\;\;\text{is closed;}
\end{equation*}
consequently, by, e.g., \cite[Corollary~15.35]{BC2017}, 
\begin{equation}
\label{e:210815a}
\Delta^\perp + (\{0\}\times W^\perp)\;\;\text{is closed.}
\end{equation}
Next, using \cref{e:fixtryu}, we have the equivalences 
\begin{align*}
&\hspace{-1cm}(x,y)\in\Fix\tryu\\
&\Leftrightarrow
P_Ux = P_V\big(P_Ux+y\big) = P_W\big(R_Ux-y\big)\\
&\Leftrightarrow
P_Ux\in Z 
\;\land\; 
y\in V^\perp
\;\land\; 
P_Ux = P_W\big(P_Ux-P_{U^\perp}x-y\big)\\
&\Leftrightarrow
x\in Z + U^\perp 
\;\land\; 
y\in V^\perp
\;\land\; 
P_{U^\perp}x + y \in W^\perp. 
\end{align*}
Now define the linear operator 
$S\colon X\times X\to X\colon (x,y)\mapsto x+y$.
Hence 
\begin{align*}
\Fix \tryu
&= 
\menge{(x,y)\in (Z+U^\perp)\times V^\perp}{P_{U^\perp}x+y\in W^\perp}\\
&=
\menge{(z+u^\perp,v^\perp)}{z\in Z,\,u^\perp\in U^\perp,\,v^\perp\in V^\perp, 
\,u^\perp + v^\perp\in W^\perp}\\
&=(Z\times\{0\})
\oplus \big((U^\perp\times V^\perp) \cap S^{-1}(W^\perp)\big). 
\end{align*}
On the other hand,
$S^{-1}(W^\perp) = (\{0\}\times W^\perp)+\ker S
= (\{0\}\times W^\perp)+\Delta^\perp$ is closed by \cref{e:210815a}.
Altogether,
\begin{equation*}
\Fix \tryu
= (Z\times\{0\})
\oplus \big((U^\perp\times V^\perp) \cap 
((\{0\}\times W^\perp)+\Delta^\perp)\big), 
\end{equation*}
i.e., 
\cref{e:210815b} holds. 
Finally, \cref{e:210815c} follows from \cref{f:orthoP}.
\end{proof}

We are now ready for the main convergence result on Ryu's algorithm.

\begin{theorem} {\bf (main result on Ryu splitting)} 
Given $0<\lambda<1$ and $(x_0,y_0)\in X\times X$, generate
the sequence $(x_k,y_k)_\kkk$ via\footnote{Recall \cref{e:linM} and \cref{e:linT}
for the definitions of $M$ and $T$.}
\begin{equation*}
(\forall\kkk)\quad (x_{k+1},y_{k+1}) := (1-\lambda)(x_k,y_k)+\lambda T(x_k,y_k).
\end{equation*}
Then 
\begin{equation}
\label{e:210815e}
M(x_k,y_k)\to \big(P_Z(x_0),P_Z(x_0),P_Z(x_0)\big);
\end{equation}
in particular,
\begin{equation*}
P_U(x_k)\to P_Z(x_0).
\end{equation*}
\end{theorem}
\begin{proof}
Set $T_\lambda := (1-\lambda)\Id+\lambda T$ and observe that 
$(x_k,y_k)_\kkk = (T^k_\lambda(x_0,y_0))_\kkk$. 
Hence, by \cref{c:key} and \cref{e:210815c}
\begin{align*}
(x_k,y_k)&\to P_{\Fix T_\lambda}(x_0,y_0)=P_{\Fix T}(x_0,y_0)\\
&=(P_Zx_0,0)+P_E(x_0,y_0)\in (P_Zx_0\oplus U^\perp)\times V^\perp,
\end{align*}
where $E$ is as in \cref{firelemma}. 
Hence 
$Q_1M(x_k,y_k) = P_Ux_k \to P_U(P_Zx_0) = P_Zx_0$. 
Now \cref{e:210815d} yields
\begin{equation*}
\lim_{k\to\infty}Q_1M(x_k,y_k)=\lim_{k\to\infty}Q_2M(x_k,y_k)
= \lim_{k\to\infty}Q_3M(x_k,y_k)= P_Zx_0, 
\end{equation*}
i.e., \cref{e:210815e} and we're done.
\end{proof}

\subsection{Malitsky-Tam splitting}

\label{ss:MTlin}

Let $n\in\{3,4,\ldots\}$. 
In this subsection, we assume that 
$U_1,\ldots,U_n$ are closed linear subspaces of $X$.
We set
\begin{equation*}
(\forall i\in\{1,2,\ldots,n\})
\quad 
A_i := N_{U_i} \;\;\text{and}\;\;
P_i := P_{U_i}. 
\end{equation*}
Then
\begin{equation*}
Z := \zer(A_1+\cdots+A_n) = U_1\cap \cdots \cap U_n.
\end{equation*}
The operator $M$ defined in \cref{e:MTM} turns into
\begin{subequations}
\label{e:linMTM}
\begin{align}
M\colon X^{n-1}
&\to 
X^n\colon
\begin{pmatrix}
z_1\\
\vdots\\
z_{n-1}
\end{pmatrix}
\mapsto
\begin{pmatrix}
x_1\\
\vdots\\
x_{n-1}\\
x_n
\end{pmatrix},
\quad\text{where}\;\;
\\[+2mm]
&(\forall i\in\{1,\ldots,n\})
\;\;
x_i = \begin{cases}
P_{1}(z_1), &\text{if $i=1$;}\\
P_{i}(x_{i-1}+z_i-z_{i-1}), &\text{if $2\leq i\leq n-1$;}\\
P_{n}(x_1+x_{n-1}-z_{n -1}), &\text{if $i=n$}
\end{cases}
\end{align}
\end{subequations}
and the MT operator remains (see \cref{e:TMT}) 
\begin{equation}
\label{e:linMT}
T := \tmt \colon X^{n-1}\to X^{n-1}\colon 
\bz\mapsto 
\bz+ 
\begin{pmatrix}
(Q_2-Q_1)M\bz\\
(Q_3-Q_2)M\bz\\
\vdots\\
(Q_n-Q_{n-1})M\bz
\end{pmatrix}. 
\end{equation}

We now determine the fixed point set of the Malitsky-Tam operator. 

\begin{lemma}
\label{mtfixlemma}
The fixed point set of the MT operator $T=\tmt$ is 
\begin{align}
\label{e:210817e}
\Fix T &= \menge{(z,\ldots,z)\in X^{n-1}}{z\in Z} \oplus E,
\end{align}
where 
\begin{subequations}
\label{e:bloodyE}
\begin{align}
E &:= 
\ran\Psi \cap \big(X^{n-2}\times U_n^\perp)\\
&\subseteq 
U_1^\perp \times 
\cdots
\times (U_1^\perp+\cdots+U_{n-2}^\perp)
\times \big((U_1^\perp+\cdots+U_{n-1}^\perp)\cap U_n^\perp\big)
\end{align}
\end{subequations}
and 
\begin{subequations}
\label{e:bloodyPsi}
\begin{align}
\Psi\colon U_1^\perp \times \cdots \times U_{n-1}^\perp
&\to X^{n-1}\\
(y_1,\ldots,y_{n-1})&\mapsto 
(y_1,y_1+y_2,\ldots,y_1+y_2+\cdots+y_{n-1})
\end{align}
\end{subequations}
 is the continuous linear partial sum operator which has closed range. 

Let $\bz = (z_1,\ldots,z_{n-1})\in X^{n-1}$,
and set $\bar{z} := (z_1+z_2+\cdots+z_{n-1})/(n-1)$.
Then
\begin{equation}
\label{e:210817d}
P_{\Fix T}\bz = (P_Z\bar{z},\ldots,P_Z\bar{z}) \oplus P_E\bz \in X^{n-1}
\end{equation}
and hence 
\begin{equation}
\label{e:210817g}
P_1(Q_1P_{\Fix T})\bz = P_Z\bar{z}.
\end{equation}

\end{lemma}
\begin{proof}
Assume temporarily that $\bz\in\Fix T$ and 
set $\bx = M\bz = (x_1,\ldots,x_n)$. 
Then $\bar{x} := x_1=\cdots= x_n$ and so $\bar{x}\in Z$.
Now $P_1z_1=x_1=\bar{x}\in Z$ and thus 
\begin{equation*}
z_1 \in \bar{x}+U_1^\perp \subseteq Z + U_1^\perp.
\end{equation*}
Next, 
$\bar{x}=x_2 = P_2(x_1+z_2-z_1)=
P_2x_1 + P_2(z_2 - z_1)=
P_2\bar{x}+P_2(z_2-z_1) = \bar{x}$,
which implies $P_2(z_2-z_1)=0$ and so 
$z_2-z_1\in U_2^\perp$. 
It follows that 
\begin{equation*}
z_2 \in z_1+U_2^\perp.
\end{equation*}
Similarly, by considering $x_3,\ldots,x_{n-1}$, we obtain 
\begin{equation}
\label{e:210817b}
z_3\in z_2 + U_3^\perp, \ldots, z_{n-1}\in z_{n-2}+U_{n-1}^\perp.
\end{equation}
Finally, 
$\bar{x}=x_n=P_n(x_1+x_{n-1}-z_{n-1})=P_n(\bar{x}+\bar{x}-z_{n-1})
=2\bar{x}-P_nz_{n-1}$,
which implies $P_nz_{n-1}=\bar{x}$, i.e., 
$z_{n-1}\in \bar{x}+U_n^\perp$. 
Combining with \cref{e:210817b}, we see that 
$z_{n-1}$ satisfies
\begin{equation*}
z_{n-1}\in (z_{n-2}+U_{n-1}^\perp)\cap(P_1z_1+U_n^\perp).
\end{equation*}
To sum up, our $\bz\in \Fix T$ must satisfy
\begin{subequations}
\label{e:210817c}
\begin{align}
z_1 &\in Z + U_1^\perp\\
z_2 &\in z_1+U_2^\perp\\
&\;\;\vdots\\
z_{n-2}&\in z_{n-3}+U_{n-2}^\perp\\
z_{n-1}&\in (z_{n-2}+U_{n-1}^\perp)\cap(P_1z_1+U_n^\perp). 
\end{align}
\end{subequations}

We now show the converse. 
To this end, assume now that our $\bz$ satisfies \cref{e:210817c}. 
Note that $Z^\perp = \overline{U_1^\perp+\cdots + U_n^\perp}$.
Because $z_1 \in Z+U_1^\perp$,
there exists $z\in Z$ and $u_1^\perp\in U_1^\perp$ such that 
$z_1 = z\oplus u_1^\perp$.
Hence $x_1=P_1z_1 = P_1z = z$. 
Next, 
$z_2\in z_1+U_2^\perp$, say 
$z_2=z_1+u_2^\perp = z\oplus(u_1^\perp+u_2^\perp)$,
where $u_2^\perp \in U_2^\perp$. 
Then 
$x_2 = P_2(x_1+z_2-z_1)=P_2(z+u_2^\perp)=P_2z = z$. 
Similarly, there exists also $u_3^\perp\in U_3^\perp,
\ldots,u_{n-1}^\perp\in U_{n-1}^\perp$ such that 
$x_3=\cdots=x_{n-1}=z$ and 
$z_i=z\oplus(u_1^\perp+\cdots +u_i^\perp)$ for 
$2\leq i\leq n-1$. 
Finally, we also have $z_{n-1}=z\oplus u_n^\perp$
for some $u_n^\perp\in U_n^\perp$. 
Thus
$x_n = P_n(x_1+x_{n-1}-z_{n-1})=P_n(2z-(z+u_{n}^\perp))
=P_nz = z$. Altogether, $\bz \in \Fix T$.
We have thus verified the description of $\Fix T$ 
announced in \cref{e:210817e}, using the convenient
notation of the operator $\Psi$ which is easily seen to have closed range.

Next, we observe that 
\begin{equation}
\label{e:210817c+}
D := \menge{(z,\ldots,z)\in X^{n-1}}{z\in Z} = 
Z^{n-1}\cap \Delta,
\end{equation}
where $\Delta$ is the diagonal in $X^{n-1}$ which has projection
$P_\Delta(z_1,\ldots,z_n)=(\bar{z},\ldots,\bar{z})$ 
(see, e.g., \cite[Proposition~26.4]{BC2017}). 
By convexity of $Z$, we clearly have 
$P_\Delta(Z^{n-1})\subseteq Z^{n-1}$.
Because $Z^{n-1}$ is a closed linear subspace of $X^{n-1}$,
\cite[Lemma~9.2]{Deutsch} and \cref{e:210817c+} yield
$P_D = P_{Z^{n-1}}P_\Delta$ and therefore
\begin{equation}
\label{e:210817f}
P_D\bz
= P_{Z^{n-1}}P_\Delta\bz
= \big(P_Z\bar{z},\ldots,P_Z\bar{z}\big). 
\end{equation}

Combining \cref{e:210817e}, \cref{f:orthoP}, \cref{e:210817c+},
and \cref{e:210817f} yields \cref{e:210817d}. 

Finally, observe that $Q_1(P_E\bz)\in U_1^\perp$ by 
\cref{e:bloodyE}.
Thus $Q_1(P_{\Fix T}\bz)\in P_Z\bar{z} + U_1^\perp$ 
and \cref{e:210817g} follows.
\end{proof}

We are now ready for the main convergence result on 
the Malitsky-Tam algorithm.

\begin{theorem} {\bf (main result on Malitsky-Tam splitting)} 
Given $0<\lambda<1$ and $\bz_0=(z_{0,1},\ldots,z_{0,n-1}) \in X^{n-1}$, 
generate the sequence $(\bz_k)_\kkk$ via\footnote{Recall \cref{e:linMTM} 
and \cref{e:linMT} for the definitions of $M$ and $T$.}
\begin{equation*}
(\forall\kkk)\quad
\bz_{k+1} := (1-\lambda)\bz_k + \lambda T\bz_k.
\end{equation*}
Set 
\begin{equation*}
p := \frac{1}{n-1}\big(z_{0,1}+\cdots+z_{0,n-1}\big). 
\end{equation*}
Then there exists $\bar{\bz}\in X^{n-1}$ such that 
\begin{equation*}
\bar{\bz}_k \to \bar{\bz} \in \Fix T,
\end{equation*}
and 
\begin{equation}
\label{e:210817h}
M\bz_k \to M\bar{\bz} = (P_Zp,\ldots,P_Zp) \in X^n. 
\end{equation}
In particular, 
\begin{equation}
\label{e:210817i}
P_1(Q_1\bz_k)= Q_1M\bz_k \to P_Z(p) = \tfrac{1}{n-1}P_Z\big(z_{0,1}+\cdots+z_{0,n-1}\big).
\end{equation}
Consequently, if $x_0\in X$ and 
$\bz_0 = (x_0,\ldots,x_0)\in X^{n-1}$,
then 
\begin{equation}
\label{e:yayyay}
P_1Q_1\bz_k \to P_Zx_0.
\end{equation}
\end{theorem}
\begin{proof}
Set $T_\lambda := (1-\lambda)\Id+\lambda T$ and observe that 
$(\bz_k)_\kkk = (T_\lambda^k\bz)_\kkk$. 
Hence, by \cref{c:key} and \cref{mtfixlemma},
\begin{align*}
\bz_k&\to P_{\Fix T_\lambda}\bz_0=P_{\Fix T}\bz_0\\
&=(P_Zp,\ldots,P_Zp)\oplus P_E(\bz_0), 
\end{align*}
where $E$ is as in \cref{mtfixlemma}. 
Hence, using also \cref{e:210817g}, 
\begin{align*}
Q_1M\bz_k &= P_1Q_1\bz_k \\
&\to P_1Q_1\big((P_Zp,\ldots,P_Zp)\oplus P_E(\bz_0)\big)\\
&=P_1\big(P_Zp+Q_1(P_E(\bz_0))\big)\\
&\in P_1\big(P_Zp+U_1^\perp\big)\\
&=\{P_1P_Zp\}\\
&=\{P_Zp\},
\end{align*}
i.e., $Q_1M\bz_k\to P_Zp$. 
Now \cref{e:210817a} yields 
$Q_iM\bz_k\to P_Zp$ for every $i\in\{1,\ldots,n\}$.
This yields \cref{e:210817h} and \cref{e:210817i}.

The ``Consequently'' part is clear because when 
$\bz_0$ has this special form, then $p=x_0$. 
\end{proof}

\subsection{Campoy splitting}

\label{ss:Clin}

Let $n\in\{3,4,\ldots\}$. 
In this subsection, we assume that 
$U_1,\ldots,U_n$ are closed linear subspaces of $X$.
We set
\begin{equation*}
(\forall i\in\{1,2,\ldots,n\})
\quad 
A_i := N_{U_i} \;\;\text{and}\;\;
P_i := P_{U_i}. 
\end{equation*}
Then
\begin{equation*}
Z := \zer(A_1+\cdots+A_n) = U_1\cap \cdots \cap U_n.
\end{equation*}
By \cref{e:f:C}, 
\begin{subequations}
\begin{align}
J_{\bA}\colon (x_1,\ldots,x_{n-1})&\mapsto E\bigg(P_n\Big(
\tfrac{1}{n-1}\sum_{i=1}^{n-1}x_i\Big)\bigg)\label{e:211012b},\\
J_{\bB} \colon (x_1,\ldots,x_{n-1})&\mapsto 
\big(P_{1}x_1,\ldots,P_{n-1}x_{n-1} \big). \label{e:211012d}
\end{align}
\end{subequations}
Now recall from \cref{e:embed} that $E:x\mapsto (x,\dots,x)$ and denote 
by $\Delta=\{(x,\dots,x)\in X^{n-1}\mid x\in X\}$, which is the diagonal in $X^{n-1}$. 
We are now ready for the following result. 

\begin{lemma}
\label{l:Campoycvg}
Set $\widetilde{U} := E(U_n) = U_n^{n-1}\cap\Delta\subseteq X^{n-1}$ and
$\widetilde{V} := U_1\times\cdots\times U_{n-1}\subseteq X^{n-1}$. 
Then for every $\bz = (z_1,\ldots,z_{n-1})\in X^{n-1}$ and 
$\bar{z}=\tfrac{1}{n-1}\sum_{i=1}^{n-1}z_i$, we have 
\begin{subequations}
\label{e:211012ac}
\begin{align}
\label{e:211012a}
J_{\bA}\bz &=  P_{U_n^{n-1}}P_\Delta\bz = (P_n\bar{z},\ldots,P_n\bar{z})
= P_{\widetilde{U}}\bz,\\
J_{\bB}\bz &= P_{U_1\times\cdots\times U_{n-1}}\bz 
= P_{\widetilde{V}}\bz, \label{e:211012c}\\
 T &= \tc = \Id-2J_{\bA} + 2J_{\bB}R_{\bA}, \label{e:211012f2}
\end{align}
\end{subequations}
and 
\begin{equation}
\label{e:220201a}
\bA = N_{\widetilde{U}}
\;\;\text{and}\;\;
\bB = N_{\widetilde{V}}. 
\end{equation}
Moreover, 
$\widetilde{U}\cap \widetilde{V}= Z^{n-1}\cap\Delta$,
$P_{\Fix T} = P_{\widetilde{U}\cap \widetilde{V}} 
\oplus P_{{\widetilde{U}}^\perp \cap {\widetilde{V}}^\perp}$,
 and 
\begin{equation}
\label{e:211012h}
J_{\bA}P_{\Fix T}\bz = P_{E(Z)}\bz = E(P_Z\bar{z}). 
\end{equation}
\end{lemma}
\begin{proof}
It is clear from \cref{e:211012b} that $J_{\bA}\bz =  P_{U_n^{n-1}}P_\Delta\bz$. 
Note that $P_{U_n^{n-1}}(\Delta) \subseteq \Delta$. 
It thus follows from \cite[Lemma~9.2]{Deutsch} that 
\begin{equation}\label{e:211115a}
P_\Delta P_{U_n^{n-1}} = P_{U_n^{n-1}}P_\Delta = P_{U_n^{n-1}\cap \Delta} 
= P_{\widetilde{U}}. 
\end{equation}
and \cref{e:211012a} follows. 
The formula for \cref{e:211012c} is clear from \cref{e:211012d}.

It is clear from \cref{e:CT} and \cref{e:CF} that $\fc$ is the Douglas-Rachford operator
for the pair $(\bA,\bB)$, which has the same fixed point set as the Campoy operator $\tc$. 
In view of \cref{e:220201a}, this is the feasibility case applied to the
pair of subspaces $(\widetilde{U},\widetilde{V})$. 
Note that
\begin{equation}
\label{e:211012e}
\widetilde{U}\cap \widetilde{V} = E(Z) = Z^{n-1}\cap \Delta. 
\end{equation}
By \cite[Proposition~3.6]{BBCNPW1}, 
$\Fix T = (\widetilde{U}\cap \widetilde{V})\oplus 
({\widetilde{U}}^\perp \cap {\widetilde{V}}^\perp)$, 
\begin{align}
\label{e:CPFixT}
P_{\Fix T}= P_{\widetilde{U}\cap \widetilde{V}} 
\oplus P_{{\widetilde{U}}^\perp \cap {\widetilde{V}}^\perp},
\end{align}
and 
\begin{align}\label{e:211115b}
J_{\bA}P_{\Fix T} = P_{\widetilde{U}}P_{\Fix T} = 
P_{\widetilde{U}\cap \widetilde{V}}=P_{E(Z)} = P_{Z^{n-1}\cap \Delta} 
= P_{Z^{n-1}}P_\Delta,
\end{align}
where the rightmost identity in \cref{e:211115b}
follows 
from the same argument as in the proof of \cref{e:211115a}. 
\end{proof}

\begin{theorem} {\bf (main result on Campoy splitting)}
Given $\bz_0=(z_{0,1},\ldots,z_{0,n-1}) \in X^{n-1}$ and
$0<\lambda<1$, 
generate the sequence $(\bz_k)_\kkk$ via\footnote{Recall \cref{e:211012f2} 
for the definition of $T$.}
\begin{equation*}
(\forall\kkk)\quad
\bz_{k+1} := (1-\lambda)\bz_k + \lambda T\bz_k.
\end{equation*}
Set 
\begin{equation*}
\bar{z} := \frac{1}{n-1}\big(z_{0,1}+\cdots+z_{0,n-1}\big). 
\end{equation*}
Then 
\begin{equation}
\label{e:211012f}
{\bz}_k \to P_{\Fix T}\bar{\bz}_0 \in \Fix T
\end{equation}
and 
\begin{equation}
\label{e:211012g}
M\bz_k = J_{\bA}\bz_k \to J_{\bA}P_{\Fix T}{\bz}_0 = (P_Z\bar{z},\ldots,P_Z\bar{z}) \in X^n. 
\end{equation}
\end{theorem}
\begin{proof}
Because $T$ is nonexpansive (see \cref{f:C}), we see that 
\cref{e:211012f} follows from \cref{c:key}. 
Finally, \cref{e:211012g} follows from \cref{e:211012h}. 
\end{proof}

For future use in \cref{sec:matrix}, we note that 
the operators defined in \cref{e:CampoyM} and \cref{e:CampoyS} 
 turn into
\begin{subequations}
\begin{align}
\label{e:CMlin}M
&: 
X^{n-1}\to X^{n-1} : 
\bz = 
\begin{pmatrix}
z_{1}\\
\vdots\\
z_{n-1}
\end{pmatrix}
\mapsto
\frac{1}{n-1}\begin{pmatrix}\sum_{i=1}^{n-1}P_{n}(z_{i})\\
\vdots\\
\sum_{i=1}^{n-1}P_{n}(z_{i})\end{pmatrix}
\quad\text{and}\;\;\\[+2mm]
\label{e:CSlin}
S&:X^{n-1}\to X^{n-1}:
\bz = \begin{pmatrix}
z_{1}\\
\vdots\\
z_{n-1}
\end{pmatrix}
\mapsto
\begin{pmatrix}
x_{1}\\
\vdots\\
x_{n-1}
\end{pmatrix},
\quad\text{where}\;\;
\\[+2mm]
&(\forall i\in\{1,\ldots,n\})
\;\;
x_{i} = P_{i}(2Q_iM\bz- z_{i}), 
\end{align}
\end{subequations}
while the \emph{Campoy operator} remains (see \cref{e:CampoyT}) 
 \begin{equation}
 \label{e:CTlin}
 T = \tc \colon X^{n-1}\to X^{n-1}\colon 
 \bz\mapsto 
 \bz+ 2S\bz - 2M\bz. 
 \end{equation}

\subsection{Extension to the consistent affine case}

In this subsection, we comment on the behaviour of the above
splitting algorithms  in the consistent affine case.
To this end, we shall assume that 
$V_1,\ldots,V_n$ are closed affine subspaces of $X$ with
nonempty intersection:
\begin{equation*}
V := V_1\cap V_2\cap \cdots \cap V_n \neq \varnothing.
\end{equation*}
We repose the problem of finding a point in $Z$ as 
\begin{equation*}
\text{find $x\in X$ such that $0\in (A_1+A_2+\cdots+A_n)x$,}
\end{equation*}
where each $A_i = N_{V_i}$.
When we consider Ryu splitting, we also impose $n=3$.
Set $U_i:=V_i-V_i$, which is the \emph{parallel space} of $V_i$.
Now let $v\in V$.
Then $V_i = v+U_{i}$ and hence 
$J_{N_{V_i}}=P_{V_i} = P_{v+U_i}$ satisfies
$P_{v+U_i} = v+P_{U_i}(x-v)=P_{U_i}x+P_{U_i^\perp}(v)$.
Put differently, the resolvents from the affine problem
are translations of the the resolvents from the corresponding linear problem
which considers $U_i$ instead of $V_i$.

The construction of the operator $T\in \{\tryu,\tmt,\tc\}$ now shows 
that it is a translation of the corresponding operator from the linear problem.
And finally $T_\lambda = (1-\lambda)\Id+\lambda T$ is a translation of 
the corresponding operator from the linear problem which we denote by $L_\lambda$:
$L_\lambda = (1-\lambda)\Id+\lambda L$, where $L$ is 
the Ryu operator, the Malitsky-Tam operator, or the Campoy operator 
of the parallel linear problem,
and  there exists $b\in X^{n-1}$ such that 
\begin{equation*}
T_\lambda(x) = L_\lambda(x)+b. 
\end{equation*}
By \cref{f:BLM} (applied in $X^{n-1}$), 
there exists a vector $a\in X^{n-1}$ such that 
\begin{equation}
\label{e:para1}
(\forall\kkk)\quad T_\lambda^kx = a + L_\lambda^k(x-a).
\end{equation}
In other words, the behaviour in the affine case is essentially
the same as in the linear parallel case, appropriately shifted by the vector $a$.
Moreover, because $L_\lambda^k\to P_{\Fix L}$ in the parallel linear setting,
we deduce from \cref{f:BLM} that 
\begin{equation*}
T_\lambda^k \to P_{\Fix T}
\end{equation*}
By \cref{e:para1}, the rate of convergence in the affine case are identical
to the rate of convergence in the parallel linear case.

Each of our three algorithms under consideration features an operator $M$ --- 
see \cref{e:linM}, \cref{e:linMTM}, \cref{e:CMlin} --- 
for Ryu splitting, for Malitksy-Tam splitting, for Campoy splitting,
respectively. In all cases, the convergence results established guarantee that 
\begin{equation*}
\overline{M}T^k_\lambda\bz_0 \to P_V\overline{z}, 
\end{equation*}
where $\overline{M}$ returns the arithmetic average of the output of $M$, 
where $\bz_0$ is the starting point, and  
where $\overline{z}$ is either the first component of $\bz_0$(for Ryu splitting) 
or  the arithmetic average of the components of $\bz_0$ (for Malitsky-Tam and for 
Campoy splitting); 
see
\cref{e:210815e},
\cref{e:210817i}, and
\cref{e:211012g}.

To sum up this subsection, we note that 
\emph{in the consistent affine case, Ryu's, the Malitsky-Tam, and Campoy's algorithm
each exhibits the same pleasant convergence behaviour as their linear parallel counterparts!}

It is, however, less clear how these algorithms  behave  when $V=\varnothing$. 

\section{Matrix representation}

\label{sec:matrix}

In this section, we assume that $X$ is finite-dimensional,
say 
\begin{equation*}
X=\RR^d.
\end{equation*}
All three splitting algorithms considered in this paper are of the form 
\begin{equation}
\label{e:210818a}
T_\lambda^k\to P_{\Fix T}, 
\quad\text{where $0<\lambda<1$ and $T_\lambda = (1-\lambda)\Id+\lambda T$.}
\end{equation}
{\color{black} In anticipation of \cref{sec:numexp}, 
we will also consider a fourth operator based on ``POCS'', the (sequential) method of Projections Onto Convex Sets.}
Starting from \cref{sec:main}, we have dealt with a special case where 
$T$ is a linear operator; hence, so is $T_\lambda$ and 
by \cite[Corollary~2.8]{BLM}, the convergence of the iterates 
is \emph{linear} because $X$ is finite-dimensional. 
What can be said about this rate?
By \cite[Theorem~2.12(ii) and Theorem~2.18]{BBCNPW2}, a (sharp) 
\emph{lower bound} for the rate of linear 
convergence is the \emph{spectral radius} of $T_\lambda - P_{\Fix T}$, i.e., 
\begin{equation*}
\rho\big(T_\lambda- P_{\Fix T}\big) :=
\max \big|\{\text{(possibly complex) eigenvalues of $T_\lambda- P_{\Fix T}$}\}\big|,
\end{equation*}
while an \emph{upper bound} is the operator norm
\begin{equation*}
\big\|T_\lambda- P_{\Fix T}\big\|. 
\end{equation*}
The lower bound is optimal and close to the true rate of convergence, 
see \cite[Theorem~2.12(i)]{BBCNPW2}.
Both spectral radius and operator norms of matrices are available
in programming languages such as \texttt{Julia} \cite{Julia} which features 
strong numerical linear algebra capabilities.
In order to compute these bounds for the linear rates, 
we must provide \emph{matrix representations} 
for $T$ (which immediately gives rise to one for $T_\lambda$) 
and for $P_{\Fix T}$.
In the previous sections, we casually switched back and forth
being column and row vector representations for readability.
In this section, we need to get the structure of the objects right.
To visually stress this, we will use \emph{square brackets} for vectors and 
matrices.

For the remainder of this section, we fix 
three linear subspaces $U,V,W$ of $\RR^d$, with intersection 
\begin{equation*}
Z := U\cap V\cap W. 
\end{equation*}
We assume that the matrices $P_U,P_V,P_W$ in $\RR^{d\times d}$ are available to us
(and hence so are $P_{U^\perp},P_{V^\perp},P_{W^\perp}$ and $P_Z$,
via \cref{ex:perp} and \cref{c:AD3}, respectively).

\subsection{Ryu splitting}

\label{ss:Ryublock}

In this subsection, we consider Ryu splitting.
First, 
the block matrix representation of the operator $M$ occurring 
in Ryu splitting (see \cref{e:linM}) is 
\begin{equation}
M = \begin{bmatrix}
P_U \;&  0 \\[0.5em]
P_VP_U \;& P_V\\[0.5em]
P_WP_U+P_WP_VP_U-P_W\;\; & P_WP_V-P_W
\end{bmatrix}\in\RR^{3d\times 2d}.
\label{e:RyuMmat}
\end{equation}
Hence, using \cref{e:linT}, 
we obtain the following matrix representation of the
Ryu splitting operator $T=\tryu$:
\begin{subequations}
\label{e:220525a}
\begin{align}
T &=\textcolor{black}{\begin{bmatrix}\Id\;&0\\[0.5em]0\;&\Id\end{bmatrix}+}
\begin{bmatrix}
-\Id & 0 &\Id \\[0.5em]
0 & -\Id & \Id
\end{bmatrix}
\begin{bmatrix}
P_U \;&  0 \\[0.5em]
P_VP_U \;& P_V\\[0.5em]
P_WP_U+P_WP_VP_U-P_W\;\; & P_WP_V-P_W
\end{bmatrix}\\[1em]
&= 
\begin{bmatrix}
\textcolor{black}{\Id}-P_U+P_WP_U+P_WP_VP_U-P_W &\;\; P_WP_V-P_W\\[0.5em]
P_WP_U+P_WP_VP_U-P_W-P_VP_U & \;\; \textcolor{black}{\Id+}P_WP_V-P_V-P_W
\end{bmatrix}\in\RR^{2d\times 2d}. 
\end{align}
\end{subequations}
Next, we set, as in \cref{firelemma},
\begin{subequations}
\begin{align}
\Delta &= \menge{[x,x]^\intercal\in \RR^{2d}}{x\in X},\\
E&=\big(U^\perp\times V^\perp) \cap \big(\Delta^\perp+(\{0\}\times W^\perp) \big)\label{e:RyuE}
\end{align}
\end{subequations}
so that, by \cref{e:210815c},
\begin{equation}
\label{e:RyuE4}
P_{\Fix T}\begin{bmatrix}x\\y\end{bmatrix} = 
\begin{bmatrix}P_{Z}x\\0\end{bmatrix}+ P_E\begin{bmatrix}x\\y\end{bmatrix}.
\end{equation}
With the help of \cref{c:AD3}, we see that the first term, $[P_{Z}x,0]^\intercal$, 
is obtained by applying the matrix 
\begin{equation}
\label{e:RyuE1}
\begin{bmatrix}
P_Z & 0 \\
0 & 0 
\end{bmatrix}
= 
\begin{bmatrix}
4P_U(P_U+P_V)^\dagger P_V\big(2P_U(P_U+P_V)^\dagger P_V+P_W\big)^\dagger P_W & 0 \\
0 & 0
\end{bmatrix}
\in\RR^{2d\times 2d}
\end{equation}
to $[x,y]^\intercal$. 
Let's turn to $E$, which is an intersection of two linear subspaces.
The projector of the left linear subspace making up this intersection, $U^\perp\times V^\perp$, 
has the matrix representation
\begin{equation}
\label{e:RyuE3}
P_{U^\perp \times V^\perp} = 
\begin{bmatrix}
\Id-P_U & 0 \\
0 & \Id-P_V
\end{bmatrix}. 
\end{equation}
We now turn to the right linear subspace, 
$\Delta^\perp+(\{0\}\times W^\perp)$, which is a sum
of two subspaces whose complements are
$\Delta^{\perp\perp}=\Delta$ and
$((\{0\}\times W^\perp)^\perp = X\times W$, respectively.
The projectors of the last two subspaces are
\begin{equation*}
P_\Delta =
\frac{1}{2}
\begin{bmatrix}
\Id & \Id \\
\Id & \Id
\end{bmatrix}
\;\;\text{and}\;\;
P_{X\times W} =
\begin{bmatrix}
\Id & 0 \\
0 & P_W
\end{bmatrix},
\end{equation*}
respectively.
Thus, \cref{c:ADsum} yields
\begin{subequations}
\label{e:RyuE2}
\begin{align}
&P_{\Delta^\perp+(\{0\}\times W^\perp)}\\
&= 
\begin{bmatrix}
\Id & 0 \\
0 & \Id
\end{bmatrix}
- 2\cdot \frac{1}{2}
\begin{bmatrix}
\Id & \Id \\
\Id & \Id
\end{bmatrix}
\left(
\frac{1}{2}
\begin{bmatrix}
\Id & \Id \\
\Id & \Id
\end{bmatrix}
+
\begin{bmatrix}
\Id & 0 \\
0 & P_W
\end{bmatrix}
\right)^\dagger
\begin{bmatrix}
\Id & 0 \\
0 & P_W
\end{bmatrix}\\
&= 
\begin{bmatrix}
\Id & 0 \\
0 & \Id
\end{bmatrix}
- 2
\begin{bmatrix}
\Id & \Id \\
\Id & \Id
\end{bmatrix}
\begin{bmatrix}
{3}\Id & \Id \\
\Id & \Id+2P_W
\end{bmatrix}^\dagger
\begin{bmatrix}
\Id & 0 \\
0 & P_W
\end{bmatrix}. 
\end{align}
\end{subequations}
To compute $P_E$, where $E$ is as in \cref{e:RyuE}, 
we combine \cref{e:RyuE3}, \cref{e:RyuE2} under the umbrella
of \cref{f:AD} --- the result does not seem to simplify so 
we don't typeset it. 
Having $P_E$, we simply add it to \cref{e:RyuE1} to obtain 
$P_{\Fix T}$ because of \cref{e:RyuE4}.

\subsection{Malitsky-Tam splitting}

\label{ss:MTblock}

In this subsection, we turn to Malitsky-Tam splitting for the current setup
--- this corresponds to \cref{ss:MTlin} with $n=3$ and where
we identify $(U_1,U_2,U_3)$ with $(U,V,W)$.
The block matrix representation of $M$ from \cref{e:linMTM} is 
\begin{equation}
\begin{bmatrix}
P_U \;&  0 \\[0.5em]
-P_V(\Id-P_U) \;& P_V\\[0.5em]
P_W(P_U+P_VP_U-P_V)\;\; & -P_W(\Id-P_V)
\end{bmatrix}
\in\RR^{3d\times 2d}.
\label{e:MTMmat}
\end{equation}
Thus, using \cref{e:linMT}, 
we obtain the following matrix representation 
of the Malitsky-Tam splitting operator $T=\tmt$: 
\begin{subequations}
\label{e:220526a}
\begin{align}
T &=
\textcolor{black}{\begin{bmatrix}\Id\;&0\\[0.5em]0\;&\Id\end{bmatrix}+}
\begin{bmatrix}
-\Id & \Id & 0 \\[0.5em]
0 & -\Id & \Id
\end{bmatrix}
\begin{bmatrix}
P_U \;&  0 \\[0.5em]
-P_V(\Id-P_U) \;& P_V\\[0.5em]
P_W(P_U+P_VP_U-P_V)\;\; & -P_W(\Id-P_V)
\end{bmatrix}\\[1em]
&= 
\begin{bmatrix}
\textcolor{black}{\Id}-P_U-P_V(\Id-P_U) &\;\; P_V\\[0.5em]
P_V(\Id-P_U)+P_W(P_U+P_VP_U-P_V) & \;\; \textcolor{black}{\Id}-P_V-P_W(\Id-P_V)
\end{bmatrix}\\[1em]
&= 
\begin{bmatrix}
(\Id-P_V)(\Id-P_U) &\;\; P_V\\[0.5em]
(\Id-P_W)P_V(\Id-P_U)+P_WP_U & \;\; (\Id-P_W)(\Id-P_U)
\end{bmatrix}\in\RR^{2d\times 2d}. 
\end{align}
\end{subequations}
Next, in view of \cref{e:210817d}, 
we have 
\begin{equation}
\label{e:zahn4}
P_{\Fix T}
= \frac{1}{2}
\begin{bmatrix}
P_Z & P_Z \\
P_Z & P_Z
\end{bmatrix}
+ P_E,
\end{equation}
where (see \cref{e:bloodyE} and \cref{e:bloodyPsi})
\begin{equation}
\label{e:zahn0}
E = \ran\Psi \cap (X\times W^\perp)
\end{equation}
and 
\begin{equation*}
\Psi \colon U^\perp \times V^\perp \to X^2
\colon \begin{bmatrix} y_1\\y_2\end{bmatrix}
\mapsto \begin{bmatrix} y_1\\y_1+y_2\end{bmatrix}. 
\end{equation*}
We first note that 
\begin{equation*}
\ran \Psi = \ran 
\begin{bmatrix}
\Id & 0 \\
\Id & \Id
\end{bmatrix}
\begin{bmatrix}
P_{U^\perp} & 0 \\
0 & P_{V^\perp}
\end{bmatrix}
=\ran
\begin{bmatrix}
P_{U^\perp} & 0 \\
P_{U^\perp} & P_{V^\perp}
\end{bmatrix}. 
\end{equation*}
We thus obtain from \cref{f:Pran} that 
\begin{equation}
\label{e:zahn1}
P_{\ran\Psi} = \begin{bmatrix}
P_{U^\perp} & 0 \\
P_{U^\perp} & P_{V^\perp}
\end{bmatrix}
\begin{bmatrix}
P_{U^\perp} & 0 \\
P_{U^\perp} & P_{V^\perp}
\end{bmatrix}^\dagger.
\end{equation}
On the other hand,
\begin{equation}
\label{e:zahn2}
P_{X\times W^\perp}
= \begin{bmatrix}
\Id & 0 \\
0 & P_{W^\perp}
\end{bmatrix}. 
\end{equation}
In view of \cref{e:zahn0}
and \cref{f:AD}, we obtain
\begin{equation}
\label{e:zahn3}
P_E = 2P_{\ran\Psi}\big( P_{\ran\Psi}+P_{X\times W^\perp}\big)^\dagger
P_{X\times W^\perp}.
\end{equation}
We could 
now use our formulas \cref{e:zahn1} and \cref{e:zahn2}
for $P_{\ran\Psi}$ and $P_{X\times W^\perp}$
to obtain a more explicit formula for $P_E$ --- 
but we refrain from doing so as the expressions become unwieldy.
Finally, plugging the formula for $P_Z$ from \cref{c:AD3} 
into \cref{e:zahn4} as well as plugging \cref{e:zahn3} into \cref{e:zahn4} 
yields a formula for $P_{\Fix T}$.

\subsection{Campoy splitting}

In this subsection, we look at the Campoy splitting for the current setup
--- this corresponds to \cref{ss:MTlin} with $n=3$ and where
we identify $(U_1,U_2,U_3)$ with $(U,V,W)$.

Using the linearity of $P_W$, we see that 
the block matrix representation of $M$ from \cref{e:CMlin} is 
\begin{equation}\label{e:CMmat}
M=\frac12\begin{bmatrix}
P_W&P_W\\P_W&P_W
\end{bmatrix}. 
\end{equation}
We then write $S$ from \cref{e:CSlin} as
\begin{align*}
S&=\begin{bmatrix}
P_U&0\\
0&P_V
\end{bmatrix}\parens*{2M
-\begin{bmatrix}
\Id&0\\0&\Id
\end{bmatrix}}\\
&=\begin{bmatrix}
P_U&0\\
0&P_V
\end{bmatrix}{\begin{bmatrix}
P_W-\Id&P_W\\P_W&P_W-\Id
\end{bmatrix}}\\
&=\begin{bmatrix}
P_UP_W-P_U&P_UP_W\\P_VP_W&P_VP_W-P_V
\end{bmatrix}. 
\end{align*}
Therefore, using \cref{e:CTlin}, we see that 
the Campoy operator $T=T_C$ is expressed as 
\begin{subequations}
\label{e:220526c}
\begin{align}
T 
&=\begin{bmatrix}
\Id&0\\0&\Id\end{bmatrix}
+2S 
-2M\\
&= \begin{bmatrix}
\Id&0\\0&\Id\end{bmatrix}
+ 2\begin{bmatrix}
P_UP_W-P_U&P_UP_W\\P_VP_W&P_VP_W-P_V
\end{bmatrix}
- \begin{bmatrix}P_W & P_W\\P_W & P_W\end{bmatrix}
\\
&=\begin{bmatrix}
\Id+2P_UP_W-2P_U - P_W& 2P_UP_W-P_W\\
2P_VP_W- P_W&\Id+2P_VP_W-2P_V-P_W
\end{bmatrix}. 
\end{align}
\end{subequations}
We set, as in \cref{l:Campoycvg}, 
\begin{equation*}
\widetilde{U} := W^2 \cap \Delta
\;\;\text{and}\;\;
\widetilde{V} := U\times V,
\end{equation*}
where $\Delta$ is the diagonal in $X^2$, and we obtain from \cref{e:211012ac}
\begin{align}
\label{e:211116a}
P_{\widetilde{U}}=
P_{W^2}P_\Delta = 
\begin{bmatrix}
P_W & 0 \\ 0 & P_W
\end{bmatrix}
\frac{1}{2}
\begin{bmatrix}
\Id & \Id \\ \Id & \Id 
\end{bmatrix}
=\frac{1}{2}
\begin{bmatrix}
P_W & P_W \\ P_W & P_W
\end{bmatrix}
\quad\text{and}\quad P_{\widetilde{V}}=\begin{bmatrix}
P_U & 0 \\ 0 & P_V
\end{bmatrix}.
\end{align}
We now use these projection formulas along with 
\cref{l:Campoycvg}, \cref{f:AD}, and \cref{ex:perp} to obtain 
\begin{subequations}
\label{e:211116b}
\begin{align}
P_{\Fix T}&=P_{\widetilde{U}\cap \widetilde{V}}+
P_{\widetilde{U}^\perp\cap \widetilde{V}^\perp}\\
&=2P_{\widetilde{U}}(P_{\widetilde{U}}+P_{\widetilde{V}})^\dagger P_{\widetilde{V}}+
2(\Id-P_{\widetilde{U}})(2\Id-P_{\widetilde{U}} -P_{\widetilde{V}})^\dagger
(\Id-P_{\widetilde{V}}). 
\end{align}
\end{subequations}
If desired, one may express $P_{\Fix T}$ in terms of $P_U,P_V,P_W$ 
by substituting \cref{e:211116a} into \cref{e:211116b}; 
however, due to limited space, we refrain from listing the outcome. 

{\color{black}
\subsection{POCS}

When applied to linear subspaces, it is known that the method of projections onto convex sets (POCS) converges to the projection onto the intersection. Hence POCS is a natural (sequential) algorithm to compare
the above splitting methods to. 
The standard POCS operator is $P_WP_VP_U$.
We now derive a corresponding operator $T$ based on the notion
of averaged operators 
(see, e.g., \cite[Section~4.5]{BC2017} for more on this notion).
Observe first that each projector is $\thalb$-averaged (\cite[Remark~4.34(iii)]{BC2017}). Next, using \cite[Proposition~4.46]{BC2017}, the composition
$P_WP_VP_U$ is $\tfrac{3}{4}$-averaged. We therefore set 
\begin{equation}
\label{e:RPOCS}
T := \tfrac{4}{3}P_WP_VP_U-\tfrac{1}{3}\Id.
\end{equation}
Then 
\begin{equation}
\Fix T = U\cap V\cap W = Z
\end{equation}
so $P_{\Fix T }= P_Z$ and $P_WP_VP_U = T_{3/4}$. 
Because $T$ is nonexpansive, \cref{e:210818a} yields that 
$T^k_\lambda \to P_{Z}$ when $0<\lambda<1$. 
Note that POCS operates in the original space $X$ and not in a product space like the other operators. 

\begin{remark} {\bf (sequential vs parallel)}
Consider the much simpler case of comparing alternating projections (unrelaxed POCS for two subspaces) to parallel projections. 
Combining works by Kayalar and Weinert \cite[Theorem~2]{KW} 
(see also \cite[Theorem~9.31]{Deutsch}) and 
by Badea, Grivaux, and M\"uller \cite[Proposition~3.7]{BGM} 
yields 
\begin{equation}
\|\thalb(P_U+P_V)-P_{U\cap V}\|= \thalb\|P_VP_U-P_{U\cap V}\|+\thalb,
\end{equation}
which shows that alternating projections perform better than parallel projections. Unfortunately, the situation is less clear 
for $3$ or more subspaces -- we refer the reader to \cite[Theorem~4.4.(i)]{BGM} where 
an estimate for the unrelaxed POCS operator is presented. 
However, the case of two subspaces indicates that it would not come as a surprise if a sequential methods outperforms a parallel one. 
\end{remark}

}

\section{Numerical experiments}

\label{sec:numexp}

We now describe several experiments to evaluate the performance 
of the algorithms described in \cref{sec:matrix}. 
Each instance of an experiment involves three subspaces $U_i$ of dimension $d_i$ for $i\in\{1,2,3\}$ in $X=\RR^d$. 
By \cite[equation~(4.419) on page~205]{Meyer}, 
\begin{equation*}
\dim(U_1+ U_2)=d_1+d_2-\dim(U_1\cap U_2).
\end{equation*}
Hence
\begin{equation*}
\dim(U_1\cap U_2)=d_1+d_2-\dim(U_1+U_2)\geq d_1+d_2-d.
\end{equation*}
Thus $\dim(U_1\cap U_2)\geq 1$ whenever 
\begin{equation}\label{e:U1U2dim}
d_1+d_2\geq d+1.
\end{equation}
Similarly, 
\begin{equation*}
\dim(Z)\geq \dim(U_1\cap U_2)+d_3-d\geq d_1+d_2-d+d_3-d=d_1+d_2+d_3-2d.
\end{equation*}
Along with \cref{e:U1U2dim}, a sensible choice for $d_i$ satisfies 
\begin{equation*}
d_i\geq 1+\lceil2d/3\rceil
\end{equation*}
because then $d_1+d_2\geq 2+2\lceil 2d/3\rceil\geq 2+4d/3>2+d$. 
Hence $d_1+d_2\geq 3+d$ and $d_1+d_2+d_3> 3+3\lceil 2d/3\rceil\geq 3+2d$. 
The smallest $d$ that gives proper subspaces 
is $d=6$, for which $d_1=d_2=d_3=5$ satisfy the above conditions.

We now describe our set of three numerical experiments designed to 
observe different aspects of the algorithms. 

\subsection{Experiment 1: Bounds on the rates of linear convergence}

As shown in \cref{sec:matrix}, we have lower and upper bounds on the rate of 
linear convergence of the operator $T_\lambda$. 
We conduct this experiment to observe how these bounds change as we vary 
$\lambda$.  To this end, we 
generate 1000 instances of triples of linear subspaces $(U_1,U_2,U_3)$. 
This was done by randomly generating triples of three matrices 
$(B_1,B_2,B_3)$ each drawn from $\RR^{6\times 5}$. 
These were used to define the range spaces of these subspaces, 
which in turn gave us the projection onto $U_i$ 
(using, e.g., \cite[Proposition~3.30(ii)]{BC2017}) as 
\begin{equation*}
P_{U_i}=B_iB_i^\dagger.
\end{equation*}
For each instance, algorithm, and $\lambda\in 
\tmenge{0.01\cdot k}{k\in\{1,2,\ldots,110\}}$, 
we obtain the operators $T_\lambda$ and $P_{\Fix T}$ as outlined in \cref{sec:matrix} 
and compute the spectral radius and operator norm of 
$T_\lambda-P_{\Fix T}$. 
Note that the convergence of the algorithms is only guaranteed for $\lambda\in 
\tmenge{0.01\cdot k}{k\in\{1,2,\ldots,99\}}$, 
but we have plotted beyond this range 
to observe the behaviour of the algorithms. 
\cref{fig:exp1} reports the median of the spectral radii and operator norms for each $\lambda$. 

\begin{figure}[ht!]
\centering
\includegraphics[width=\textwidth]{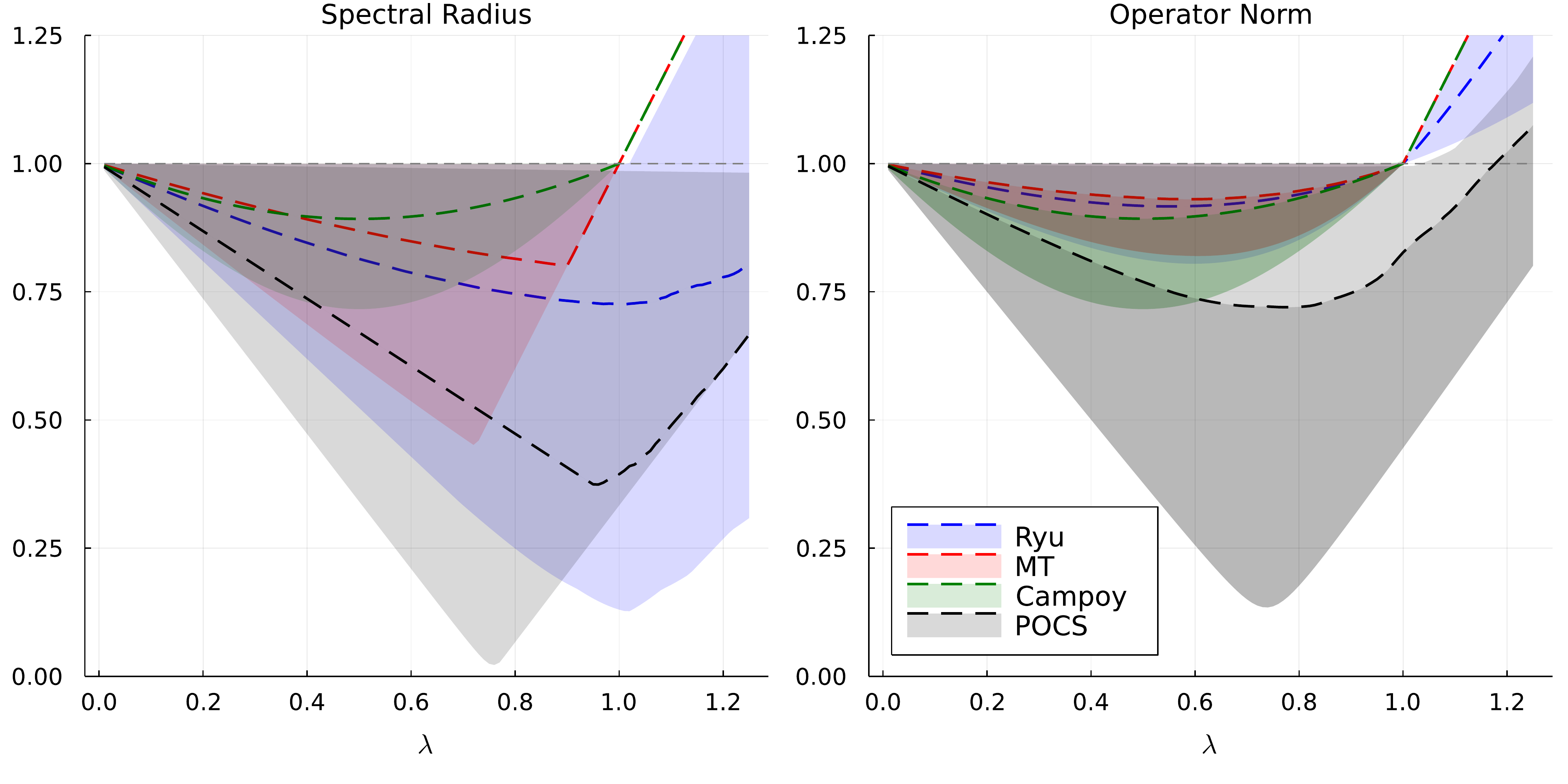}
\caption{Experiment 1: The median (solid line) and range (shaded region) of the spectral radii and operator norms 
}
\label{fig:exp1}
\end{figure}

For the spectral radius, while Ryu  shows a steady decline in the median value for $\lambda< 1$, 
Malitsky-Tam, Campoy, and POCS start increasing well before $\lambda=1$. 
The maximum value, which can be seen by the range of these values for each algorithm,
is some value less than, but close to $1$ for all $\lambda<1$.
The operator norm plot, which provides the upper bound for the convergence rates,
also stays below 1 for $\lambda<1$.
For MT, the sharp edge in the spectral norm appears as the spectral radius 
involves the maximum of different spectral values. 
By visual inspection, we see that for Campoy the lower and upper bounds coincide 
--- we weren't aware of this beforehand and are now able to provide a rigorous proof in 
\cref{r:radial} below. 

\begin{remark}{\bf (relaxed Douglas-Rachford operator)}
\label{r:radial}
Suppose that $T=\tc$ is the Campoy operator which implies 
that $T_\lambda = (1-\lambda)\Id+\lambda T$ is 
the relaxed Douglas-Rachford operator for finding a zero of 
the sum of two normal cone operators of two linear subspaces. 
By \cite[Theorem~3.10(i)]{BBCNPW2}, the matrix $T_\lambda$ is normal, i.e., 
$T_\lambda T_\lambda^* = T_\lambda^* T_\lambda$.
Moreover, \cite[Proposition~3.6(i)]{BBCNPW1} yields 
$\Fix T = \Fix T_\lambda = \Fix T_\lambda^* = \Fix T^*$. 
Set $P := P_{\Fix T} = P_{\Fix T^*}$ for brevity.
Then $P=P^*$ and $T_\lambda P = P = T_\lambda^*P$.
Taking the transpose yields $PT_\lambda^* = P = PT_\lambda$. 
Thus
\begin{align*}
(T_\lambda-P)(T_\lambda-P)^* 
&= 
T_\lambda T_\lambda^* - T_\lambda P - PT_\lambda^*+P\\
&= 
T_\lambda T_\lambda^* - P\\
&= 
T_\lambda^* T_\lambda - P\\
&=
T_\lambda^* T_\lambda - T_\lambda^*P - PT_\lambda + P\\
&=
(T_\lambda-P)^*(T_\lambda-P), 
\end{align*}
i.e., $T_\lambda-P$ is normal as well. 
Now \cite[page~431f]{GoldZwas} implies that the matrix
$T_\lambda-P$ is \emph{radial}, i.e., its spectral radius coincides 
with its operator norm. 
This explains that the two (green) curves for the Campoy algorithm in 
\cref{fig:exp1} are identical. 
\end{remark}

\subsection{Experiment 2: Number of iterations to achieve prescribed accuracy}

Because we know 
the limit points of the governing as well as shadow sequences, 
we investigate how varying $\lambda$ affects the number of iterations 
required to approximate 
the limit to a given accuracy.
For Experiment~2, we fix 100 instances of triples of subspaces $(U_1,U_2,U_3)$. 
We also fix 100 different starting points in $\RR^6$.
For each instance of the subspaces, starting point $\bz_0$ and 
$\lambda\in\tmenge{0.01\cdot k}{k\in\{1,2,\ldots,199\}}$, 
we obtain the number of iterations (up to a maximum of $10^4$ iterations) required to
achieve $\varepsilon=10^{-6}$ accuracy.

For the governing sequence, the limit $P_{\fix T}(\bz_0)$ is used to determine the stopping condition. 
\cref{fig:exp2a} reports the median number of iterations required for each $\lambda$ to achieve the given accuracy. 
\begin{figure}[H]
\centering
\includegraphics[width=\textwidth]{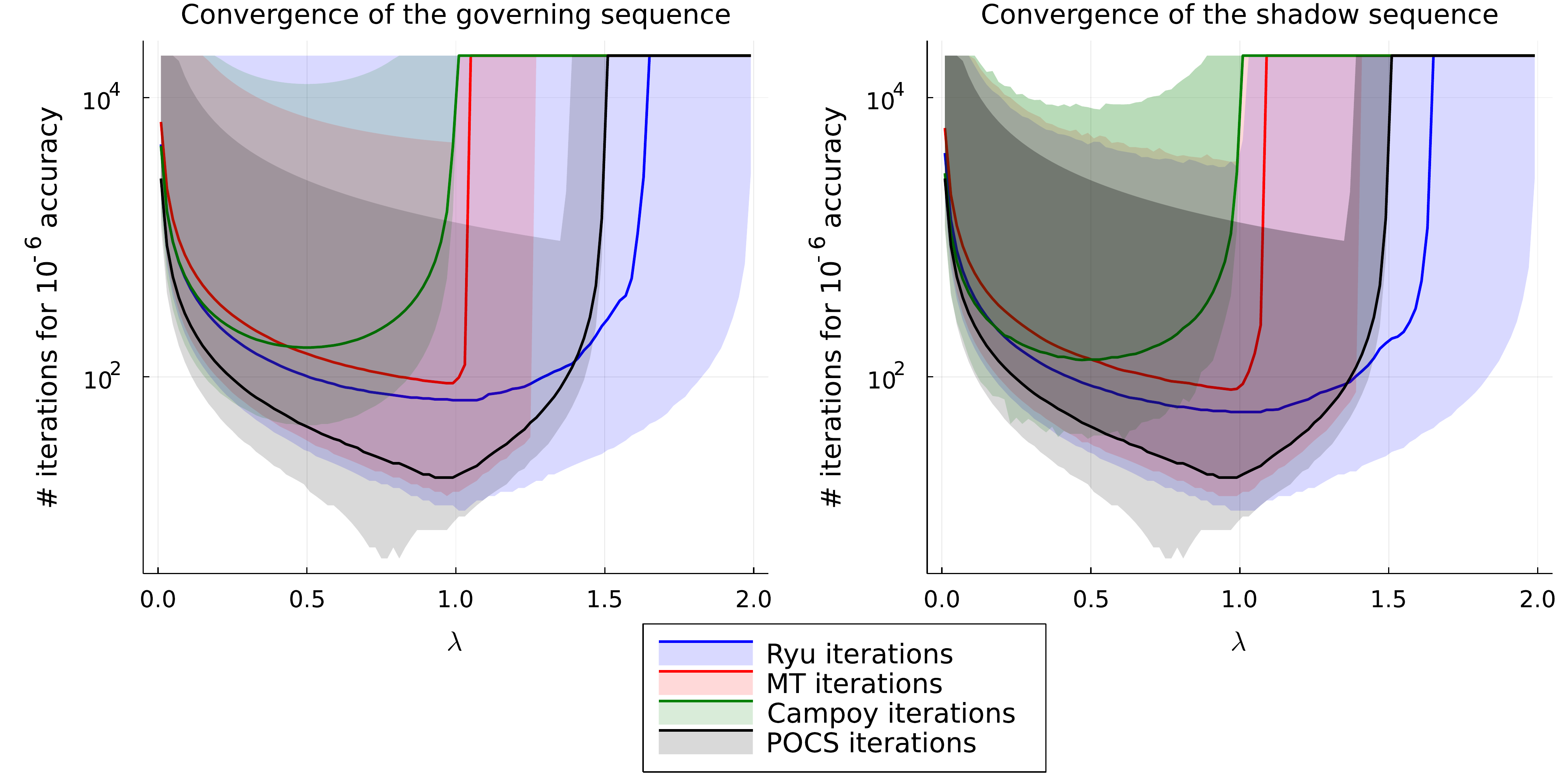}
\caption{Experiment 2: The median (solid line) and range (shaded region) of the number of iterations for the governing  and shadow sequences}
\label{fig:exp2a}
\label{fig:exp2}
\end{figure}

For the shadow sequence, we compute the median number of iterations required to achieve $\varepsilon=10^{-6}$ accuracy for the sequence $(\overline{M}\bz_k)_\kkk$ with respect to its limit $P_Zz_0$, where $z_0$ is the average of the components of $\bz_0$. 
See \cref{fig:exp2} for results. 
{\color{black} For POCS, we pick the governing sequence equal to the shadow sequence.}

For Ryu and MT in both experiments, 
increasing values of $\lambda$ result in a decreasing number of median iterations required for $\lambda<1$.
For Campoy, the median iterations reach the minimum at $\lambda\approx 0.5$ 
and keep increasing after that.
As is evident from the maximum number of iterations required for a fixed $\lambda$, 
the shadow sequence converges before the governing sequence for larger values of $\lambda$ in $\left]0,1\right[$.
One can also see that Ryu consistently requires fewer median iterations 
for both the governing and the shadow sequence to achieve the same accuracy as MT for a fixed lambda. However, Campoy performs better than Ryu for small values of $\lambda$, and beats MT for a larger range of $\lambda\in \left]0,0.5\right[$. 
{\color{black} POCS naturally performed better for all values of $\lambda$. Its behaviour for values of $\lambda$ beyond 1 is similar to Ryu.}

\subsection{Experiment 3: Convergence plots of shadow sequences}

In this experiment, we measure the distance of the terms of the governing (and shadow) sequence from its limit point,
to observe how the iterates of the algorithms approach the solution. 
Guided by \cref{fig:exp2}, 
we pick the $\lambda$ for which the median iterates are the least: 
$\lambda=0.99$ for Ryu, $\lambda=0.97$ for MT, and $\lambda=0.57$ for Campoy.
Similar to the setup of Experiment~2,
we fix 100 starting points and 100 triples of subspaces $(U_1,U_2,U_3)$. 
We then run the algorithms for 150 iterations for each starting point 
and each set of subspaces, 
and we measure the distance of the iterates $\overline{M}\bz_k$ to its limit $P_Zz_0$. 
\cref{fig:exp3} reports the median of these values for each iteration counter 
$k\in\{1,\dots,150\}$. 
{\color{black} Again note that for POCS, the governing sequence coincides with the shadow sequence.}
As can be seen in \cref{fig:exp3} for both governing and shadow sequences, Ryu converges faster to the solution compared to MT and Campoy. Ryu and MT show faint ``rippling'' as is well known to occur for the Douglas-Rachford algorithm.
Campoy, already being a DR algorithm, has more prominent ripples. 
{\color{black} As seen in the previous experiment, POCS outperforms the rest of the algorithms for its optimal value of $\lambda$.}
\begin{figure}[ht]
\centering
\includegraphics[width=\textwidth]{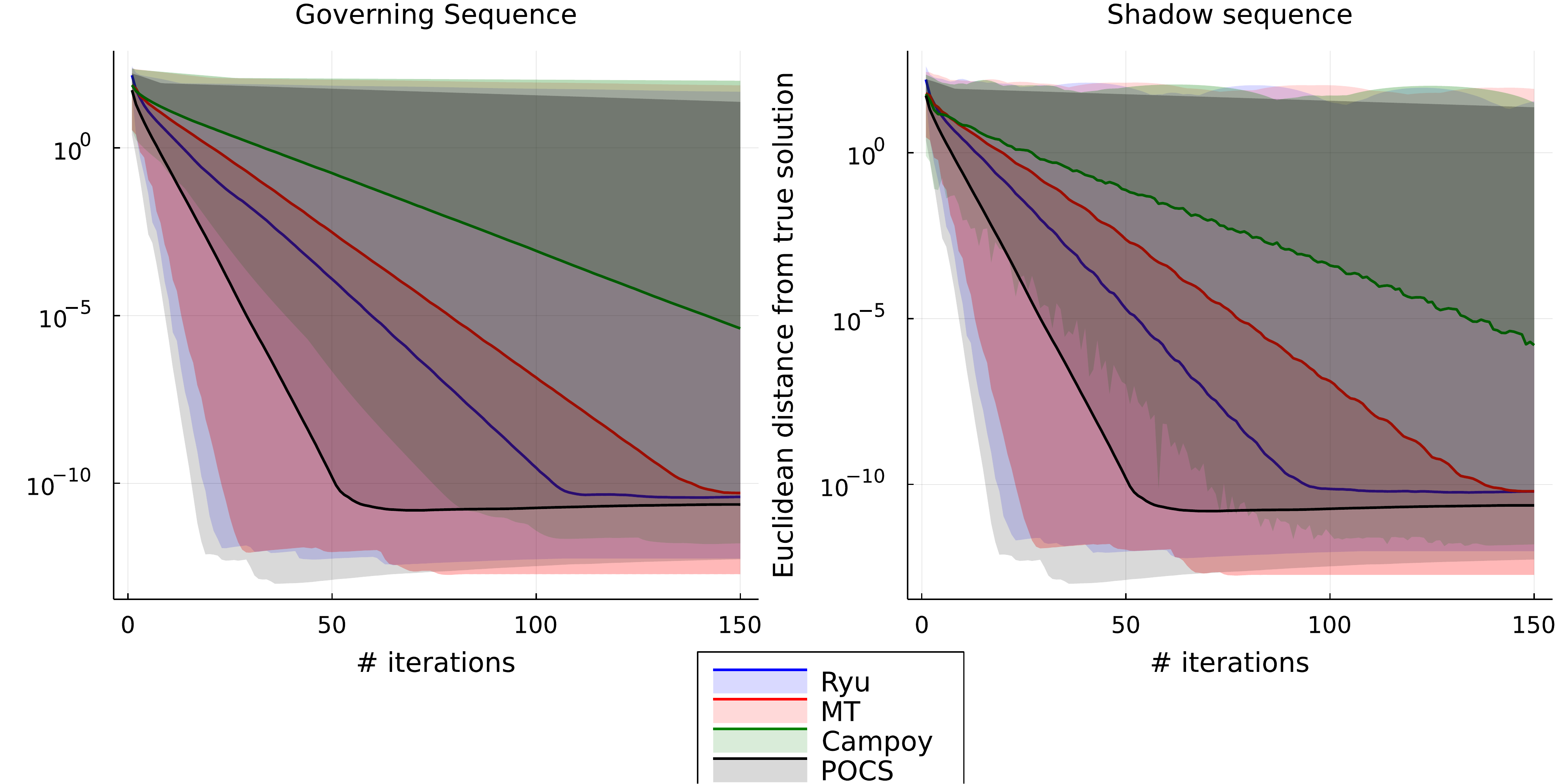}
\caption{Experiment 3: Convergence plot of the governing and shadow sequences. The median (solid line) and range (shaded region) of distances are reported.}
\label{fig:exp3}
\end{figure}

{\color{black}
\subsection{Discussion}

We tested the performance of these algorithms in the specific case of 3 subspaces in $\RR^6$ through different experiments. The performance of Ryu is better than that of MT and Campoy; however, MT and Campoy have the upper hand when it comes to extending to more subspaces since Ryu is only designed for the 3 subspaces case. Being a sequential algorithm operating solely in the original space (rather than a product space), 
POCS is faster than the three others in terms of convergence --- but the scope of these algorithms is not restricted to projection operators.
}

{\color{black}
\section{Three simple lines in the Euclidean plane}

\label{sec:3lines}

In this section, we consider a very simple situation:
Three lines passing through the origin, where the second line
bisects the angle between the first and the third.
We follow the set up in \cite[Section~5]{BBCNPW1}. 
Suppose that $X=\RR^2$, and 
set
$e_0 := [1,0]^T$, $e_{\pi/2} := [0,1]^T$, and also 
\begin{equation*}
e_\theta := \cos(\theta)e_0 + \sin(\theta)e_{\pi/2}.
\end{equation*}
Define the counterclockwise rotator
\begin{equation*}
R_\theta := 
\begin{bmatrix}
\cos(\theta) & -\sin(\theta)\\
\sin(\theta) & \cos(\theta)
\end{bmatrix}
\end{equation*}
and now set 
\begin{equation*}
U := \RR e_0, \quad 
V := \RR e_\theta = R_\theta(U), \quad
W := \RR e_{2\theta} = R_{2\theta}(U), 
\end{equation*}
where $\theta\in\left]0,\pi/2\right[$. 
It is clear that $Z=U\cap V\cap W = \{0\} = U\cap V=V\cap W=U\cap W$ 
and that the Friedrichs angle between $U$ and $V$, and between $V$ and $W$, is $\theta$ while the Friedrichs angle between $U$ and $W$ is $2\theta$. 
The projectors associated with $U,V,W$ are 
\begin{align*}
P_{U} &= 
\begin{bmatrix}
1 & 0\\
0 & 0 
\end{bmatrix},\\
P_{V} &= 
\begin{bmatrix}
\cos^2(\theta) & \sin(\theta)\cos(\theta)\\
\sin(\theta)\cos(\theta)& \sin^2(\theta)
\end{bmatrix},\\
P_{W} &= 
\begin{bmatrix}
\cos^2(2\theta) & \sin(2\theta)\cos(2\theta)\\
\sin(2\theta)\cos(2\theta)& \sin^2(2\theta)
\end{bmatrix}. 
\end{align*}

In the next few subsections, we discuss what the relevant splitting operators turn into, while in the last subsection we summarize our findings. The overall strategy is clear:
Find a formula for the splitting operator $T$ and then for $\Fix T$.
Consider $((1-\lambda)\Id+\lambda T)-P_{\Fix T}$ and determine its 
largest absolute eigenvalue and its largest spectral value to find the spectral radius and the operator norm. 
It turns out that this is easier said than done --- even with the help of SageMath \cite{Sage}.

\subsection{Ryu splitting}

The underlying $T = \tryu$ (see \cref{e:220525a}) turns out to be
{\scriptsize
\begin{align*}
T&=\begin{bmatrix}
2 \cos^4(\theta) - \cos^2(\theta) 
& -2 \big(2\cos^3(\theta) - \cos(\theta)\big)\sin(\theta) 
& -2\sin^4(\theta) + \sin^2(\theta) 
& -\big(2 \cos^3(\theta) - \cos(\theta)\big)\sin(\theta) \\
2\cos^3(\theta)\sin(\theta) 
& -4 \cos^2(\theta) \sin^2(\theta) + 1 
& 2 \cos(\theta)\sin^3(\theta) 
& 2 \cos^4(\theta) - 2 \cos^2(\theta) \\
2 \cos^4(\theta) - 2 \cos^2(\theta) 
& -2 \big(2 \cos^3(\theta) - \cos(\theta)\big) \sin(\theta) 
& -2 \cos^4(\theta) + 2 \cos^2(\theta) 
& -2 \cos^3(\theta)\sin(\theta) \\
\big(2 \cos^3(\theta) - \cos(\theta)\big) \sin(\theta) 
& -4 \cos^2(\theta) \sin^2(\theta) 
& -\big(2 \cos^3(\theta) - \cos(\theta)\big) \sin(\theta) 
& 2 \cos^4(\theta) - \cos^2(\theta) 
\end{bmatrix}.
\end{align*}}
Following the procedure of \cref{ss:Ryublock}, one eventually arrives 
at 
\begin{equation*}
P_{\Fix T} = 
\begin{bmatrix} 0 & 0 & 0 & 0 \\ 
0 & -\frac{1}{4 \sin^2(\theta) - 5} 
& \frac{2 \cos(\theta) \sin(\theta)}{4\cos^2(\theta) + 1} 
& -\frac{2 \big(\cos^4(\theta) - \cos^2(\theta)\big)}{4  \sin^4(\theta) - 5\sin^2(\theta)} \\ 
0 & \frac{2 \cos(\theta)\sin(\theta)}{4 \cos^2(\theta) + 1} 
& \frac{4 \big(\sin^4(\theta) - \sin^2(\theta)\big)}{4  \sin^2(\theta) - 5} 
& \frac{4 \big(\cos^5(\theta) - \cos^3(\theta)\big)}{\big(4  \cos^2(\theta) + 1\big) \sin(\theta)} \\ 
0 & -\frac{2 \big(\cos^4(\theta) - \cos^2(\theta)\big)}{4  \sin^4(\theta) - 5  \sin^2(\theta)} 
& \frac{4 \big(\cos^5(\theta) - \cos^3(\theta)\big)}{\big(4  \cos^2(\theta) + 1\big) \sin(\theta)} 
& -\frac{4 \cos^4(\theta)}{4 \sin^2(\theta) - 5} \end{bmatrix}.
\end{equation*}
Now set $T_\lambda := (1-\lambda)\Id+\lambda T$. 
We are again interested in the spectral radius and the operator norm of $T_\lambda-P_{\Fix T}$ --- unfortunately, even with the help of SageMath \cite{Sage}, we were not able to compute these quantities. 
In fact, even the assignment $\lambda=\thalb$ did not lead to
manageable expressions.

\subsection{Malitsky-Tam splitting}

Using \cref{e:220526a}, we find that 
$T=\tmt$ is 
\begin{equation*}
T = \begin{bmatrix}
0 & -\cos(\theta) \sin(\theta) & \cos^2(\theta) 
& \cos(\theta) \sin(\theta) \\
0 & \cos^2(\theta) & \cos(\theta) \sin(\theta) 
& \sin^2(\theta) \\
4 \cos^4(\theta) - 4 \cos^2(\theta) + 1 
& 2 \cos(\theta) \sin^3(\theta) & 0 
& -2 \big(2 \cos^3(\theta) - \cos(\theta)\big) \sin(\theta) \\
2 \big(2 \cos^3(\theta) - \cos(\theta)\big) \sin(\theta) 
& 2 \sin^4(\theta) - \sin^2(\theta) & 0 
& -4  \cos^2(\theta) \sin^2(\theta) + 1
\end{bmatrix}. 
\end{equation*}
Following the recipe outlined in \cref{ss:MTblock} yields
\begin{align*}
P_{\Fix T} &= 
\begin{bmatrix} 
0 & 0 & 0 & 0 \\ 0 
& \thalb & \cos(\theta)\sin(\theta) 
& \sin^2(\theta) - \thalb \\ 
0 & \cos(\theta) \sin(\theta) 
& -2 \sin^4(\theta) + 2 \sin^2(\theta) 
& 2 \cos(\theta)\sin^3(\theta) - \cos(\theta)\sin(\theta) \\ 
0 & \sin^2(\theta) - \thalb 
& 2 \cos(\theta)\sin^3(\theta) - \cos(\theta)\sin(\theta)
& 2 \sin^4(\theta) - 2 \sin^2(\theta) + \thalb
\end{bmatrix}. 
\end{align*}
Now set $T_\lambda := (1-\lambda)\Id+\lambda T$. 
We are again interested in the spectral radius and the operator norm of $T_\lambda-P_{\Fix T}$. 
Unfortunately, even with the help of SageMath \cite{Sage}, we were not able to compute eigenvalues or the operator norm, even when we specialized to $\lambda=\thalb$.

\subsection{Campoy splitting}
Using \cref{e:220526c}, we compute the Campoy operator $T=\tc$ to be 
\begin{align*}
T = 
\begin{bmatrix}
-\sin^2(2\theta) 
& \sin(2\theta)\cos(2\theta)
& \cos^2(2\theta)
& \sin(2\theta)\cos(2\theta)\\
- \sin(2\theta)\cos(2\theta)
& \cos^2(2\theta)
& -\sin(2\theta)\cos(2\theta)
& -\sin^2(2\theta) \\
\cos(2\theta) 
& \sin(2\theta) & 0 & 0 \\
0 & 0 & -\sin(2\theta) 
& \cos(2\theta)
\end{bmatrix}.
\end{align*}
The operator $T$ is actually an isometry (as can be checked
directly or after applying \cite[Proposition~3.5]{BC2017});
unfortunately, we could not determine
$P_{\Fix T}$ because the computation of the symbolic computation of
Moore-Penrose inverses in \cref{e:211116b} turned out to be too complicated even when using SageMath \cite{Sage}.

\subsection{POCS}

The method of Projections onto Convex Sets (POCS) is known to converge to the projection onto the intersection when applied to linear subspaces. In contrast to the parallel splitting methods discussed above, POCS is sequential in nature not requiring any product space. 
First, we note that 
\begin{align*}
P_{W}P_{V}P_{U}
&= 
\begin{bmatrix} 2  \cos^4(\theta) - \cos^2(\theta) & 0 \\
2 \cos^3(\theta)\sin(\theta) & 0
\end{bmatrix}
= 
\cos^2(\theta)
\begin{bmatrix} 2\cos^2(\theta) - 1 & 0 \\
2 \cos(\theta)\sin(\theta) & 0
\end{bmatrix}\\
&=
\cos^2(\theta)
\begin{bmatrix} \cos(2\theta) & 0 \\
\sin(2\theta) & 0
\end{bmatrix}. 
\end{align*}
Next, the reflected POCS operator $T$ (see \cref{e:RPOCS}) simplifies to 
\begin{align*}
T &:= \tfrac{4}{3}P_WP_VP_U-\tfrac{1}{3}\Id
=\tfrac{1}{3}\begin{bmatrix}
{8}\cos^4({\theta}) - {4}\cos^2({\theta}) - {1} & 0 \\
{8}\cos^3({\theta}) \sin({\theta}) & -{1}
\end{bmatrix}.
\end{align*}
Note that $\Fix T=\{0\}$  and hence $P_{\Fix T}=0$. 
Thus 
Next, we set as before 
$T_\lambda := (1-\lambda)\Id + \lambda T$; hence, 
\begin{align*}
T_\lambda-P_{\Fix T}
= (1-\lambda)\Id + \lambda T. 
\end{align*}
After simplification, the two eigenvalues of $T_\lambda-P_{\Fix T}=T_\lambda$ are seen to be real
and equal to $1-\tfrac{4}{3}\lambda$ and $1+\tfrac{4}{3}\lambda\big(2\sin^4(\theta)-3\sin^2(\theta))$. 
It follows that the spectral radius of
$T_\lambda-P_{\Fix T}$ is equal to 
\begin{equation*}
\max\big\{\big|1-\tfrac{4}{3}\lambda\big|,
\big|1+\tfrac{4}{3}\lambda\big(2\sin^4(\theta)-3\sin^2(\theta)\big)\big| \big\}. 
\end{equation*}
The eigenvalues of $(T_\lambda-P_{\Fix T})^*(T_\lambda-P_{\Fix T})
= T_\lambda^*T_\lambda$ are available but too complicated to list here.
To make progress analytically, we assume that
\begin{equation*}
\lambda = \tfrac{4}{3},
\end{equation*}
i.e., we consider the original POCS operator $P_WP_VP_U$. 
The spectral radius of $T_\lambda-P_{\Fix T}$ simplifies then to 
\begin{equation*}
1+2\sin^4(\theta)-3\sin^2(\theta)=\cos^2(\theta)\cos(2\theta)
\end{equation*}
while the operator norm of $T_\lambda-P_{\Fix T}$ is 
\begin{equation*}
\cos^2(\theta). 
\end{equation*}

\subsection{Discussion}

It is tempting to ask whether the splitting methods by 
Ryu, by Malitksy-Tam, and by Campoy allow a complete analysis like the one available for Douglas--Rachford splitting for two subspaces as carried out in \cite{BBCNPW1}. 
Unfortunately, a symbolic analysis of
these methods seems much harder even for the shockingly simple case considered in this section. Thus at present, we are rather pessimistic about the prospect of obtaining pleasant convergence rate for this new breed of splitting methods --- but we hope that the reader will prove us wrong!

}

\section{Conclusion}

\label{sec:end}

In this paper, we investigated the recent splitting methods 
by Ryu, by Malitsky-Tam, and by Campoy 
in the context of normal cone operators
for subspaces. We discovered and proved that all three algorithms find not just 
some solution but in fact the \emph{projection} of the starting point onto the
intersection of the subspaces. Moreover, convergence of the iterates 
is \emph{strong} even in infinite-dimensional settings. 
Our numerical experiments illustrated that Ryu's method seems to converge 
faster although neither Malitsky-Tam splitting nor Campoy 
splitting is limited in its applicability to
just 3 subspaces. 

Two natural avenues for future research are the following.
Firstly, when $X$ is finite-dimensional, we know that the convergence
rate of the iterates
is linear. While we illustrated this linear convergence
numerically in this paper, it is open whether there are \emph{natural
bounds for the linear rates} in terms of some version of angle between
the subspaces involved. 
For the prototypical Douglas-Rachford splitting framework, 
this was carried out in  \cite{BBCNPW1} and \cite{BBCNPW2} 
in terms of the \emph{Friedrichs angle}.
Secondly, what can be said in the \emph{inconsistent} affine case?
Again, the Douglas-Rachford algorithm may serve as a guide to what 
the expected results and complications might be; see, e.g., \cite{BM16}. 
{\color{black} However, these topics for further 
research appear to be quite hard.}

\section*{Acknowledgments}
HHB and XW were supported by NSERC Discovery Grants. 
The authors thank Alex Kruger for making us aware of his \cite{Kruger81} and \cite{Kruger85} 
which is highly relevant for Campoy splitting. 
{\color{black}
Last but not least, we thank the editor and the anonymous reviewers for their pertinent and constructive comments. 
}

\end{document}